\documentclass[a4paper,11pt, reqno]{amsart}
\def\l@subsection{\@tocline{2}{0pt}{30pt}{5pc}{}}
\usepackage{amsfonts,amssymb,amsthm}
\usepackage[mathscr]{eucal}
\usepackage{enumitem}
\usepackage{graphicx}
\usepackage{epsfig}
\usepackage{psfrag}
\usepackage{amsfonts}
\usepackage{amsmath}
\usepackage{amsthm}
\usepackage{latexsym}
\usepackage{verbatim}
\usepackage{mathtools}

  \usepackage[usenames,dvipsnames]{xcolor}

\usepackage{hyperref}

\usepackage[active]{srcltx}
\usepackage{tikz}
\usepackage{xcolor,colortbl}
\definecolor{green}{rgb}{0.1,0.1,0.1}

\input xy
\xyoption{all}

\theoremstyle{plain}
\newtheorem{thm}{Theorem}

\newtheorem{corollary}[thm]{Corollary}
\newtheorem{lemma}[thm]{Lemma}
\newtheorem{prop}[thm]{Proposition}

\newtheoremstyle{exm}
{9pt}{9pt}{}{}{\bfseries}{}{.5em}{}
\theoremstyle{exm}
\newtheorem{exm}[thm]{Example}

\theoremstyle{definition}
\newtheorem{defin}{Definition}[section]

\newtheoremstyle{rmk}
{9pt}{9pt}{}{}{\bfseries}{}{.5em}{}
\theoremstyle{rmk}
\newtheorem{rmk}[thm]{Remark}
\newtheoremstyle{question}
{9pt}{9pt}{}{}{\bfseries}{}{.5em}{}
\theoremstyle{question}
\newtheorem{question}[thm]{Question}
\numberwithin{equation}{section}
\numberwithin{thm}{section}
\numberwithin{figure}{section}

\newcommand{\ft}{{\mathfrak{t}}}
\newcommand{\ts}{\mathfrak{t}^{*}}

\newcommand{\il}{\mathsf{i}}

\newcommand{\N}{\mathbb{N}}

\newcommand{\J}{\mathsf{J}}

\newcommand{\T}{\mathbb{T}}

\newcommand{\fix}{\M^{S^1}}
\renewcommand{\H}{\mathbb{H}}
\newcommand{\SL}{\mathrm{SL}}

\newcommand{\M}{\mathsf{M}}
\newcommand{\R}{\mathbb{R}}
\newcommand{\Z}{{\mathbb{Z}}}
\newcommand{\C}{{\mathbb{C}}}

\newcommand{\Q}{\mathbb{Q}}

\newcommand{\LS}{\mathbb{L}_{S^1}}

\newcommand{\ac}{(\M,\J,S^1)}
\newcommand{\fixed}{\M^{S^1}}
\DeclareMathOperator{\homo}{Hom}

\DeclareMathOperator{\sign}{sign}

\DeclareMathOperator{\ind}{Ind}
\DeclareMathOperator{\k0}{\mathbf{k_0}}

\newcommand{\GL}{\mathrm{GL}}

\newcommand\numberthis{\addtocounter{equation}{1}\tag{\theequation}}

\newcommand{\purge}[1]{}

\title[Rigidity of elliptic genera]{Rigidity of elliptic genera: from number theory to geometry and back.}

\author[K. Bringmann, A. Caviedes Castro, S. Sabatini, M. Schwagenscheidt]{Kathrin Bringmann, Alexander Caviedes Castro, \\Silvia Sabatini, Markus Schwagenscheidt}

\thanks{This research was supported by the SFB-TRR 191 \emph{Symplectic
Structures in Geometry, Algebra and Dynamics}, funded by the DFG}

\address{Department Mathematik/Informatik, Abteilung Mathematik, Universit\"at zu K\"oln, Weyertal 86-90, D-50931 K\"oln, Germany}
\email{kbringma@math.uni-koeln.de \\ caviedes@math.uni-koeln.de} \email{sabatini@math.uni-koeln.de \\ mschwage@math.uni-koeln.de}

\subjclass[2010]{58J26, 57R20, 37J15}
\keywords{Betti numbers,  circle actions,  Eisenstein series, Elliptic genera, index, Modular forms.}

\date{\today}

\begin{document}

\begin{abstract}
In this paper we derive topological and number theoretical consequences of the rigidity of elliptic genera, which are
special modular forms associated to each compact almost complex manifold.
In particular, on the geometry side, we prove that rigidity
implies relations between the Betti numbers and the index of a compact symplectic manifold
of dimension $2n$ admitting a Hamiltonian action of a circle with isolated fixed points.
We investigate the case of maximal index and toric actions.
 On the number theoretical side we prove that from each compact almost complex manifold of index greater than one, that can be endowed with the action of a circle with
 isolated fixed points, one can derive non-trivial relations among Eisenstein series. We give explicit formulas coming from the standard action on $\C P^n$.
\end{abstract}
\maketitle

\tableofcontents

\section{Introduction and statement of results}
\subsection{Motivation}
The goal of this paper is to derive topological and number theoretical consequences of the rigidity of elliptic genera on almost complex manifolds acted on
by a circle.

Let $(\M,\J)$ be a compact almost complex manifold of dimension $2n$. The elliptic genus of level $N$, denoted by $\varphi_N(\M)$,
is a certain modular form (for the group $\Gamma_1(N)$) of weight $n$ associated to $\M$.
Although its formal definition may seem convoluted (see Section \ref{def elliptic genus}), its original inspiration comes from an intuition of Witten \cite{W2} who
heuristically defined the Dirac operator on the free loop space $\mathcal{L}\M$ of $\M$. In a similar fashion, Hirzebruch justified the definition of the
elliptic genus of level $N$ as the so-called $\chi_y$-genus of $\mathcal{L}\M$ \cite[Section 7.4]{HBJ} (see also Section \ref{vvb}).

One of the surprising features of this genus is that it is rigid in the following sense.
If we assume the first Chern
class of $\M$ to be divisible by $N$, then the Fourier expansion of $\varphi_N(\M)$ at a certain cusp has coefficients given by the topological Atiyah--Singer index of
certain bundles associated to $\M$, namely tensor products of exterior and symmetric powers of $T\M$ and $T^*\M$ and of the line bundle
 $L=(\wedge^n T^*\M)^{\frac{1}{N}}$ \cite[Appendix III, Section 4]{HBJ} (see also equation \eqref{infinite tensor} and Section \ref{section rigidity}).
If $(\M,\J)$ is acted on by a circle that preserves the almost complex structure, then each of the above bundles inherits a circle action and one could
consider the  equivariant elliptic genus of level $N$, where the indices above are replaced by the equivariant indices of the bundles, which are therefore
Laurent polynomials, namely elements of $\Z[t,t^{-1}]$. The celebrated  rigidity theorem for elliptic genera \cite[Theorem on page 181]{HBJ} asserts
that if the first Chern class is divisible by $N$, then the equivariant elliptic genus of level $N$ is rigid, namely the equivariant indices of the bundles above are indeed constant.
Moreover, for certain types of actions, even more is true, as the elliptic genus vanishes identically.

The proof of the rigidity theorem has a long history, dating back to results of Landweber and Ochanine (for a detailed account see \cite{Land} and the references therein), as well
as Taubes \cite{Ta} and Bott--Taubes \cite{BT}.
Also, prior to elliptic genera, rigidity phenomena of some special bundles had already been observed by, for instance, Atiyah--Hirzebruch \cite{AH} and Hattori \cite{H2}.
However for the bundles appearing in the expansion of the elliptic genus, the unified framework offered by the latter is necessary for the proof of their rigidity.
\\$\;$
Using these deep theorems we are able to deduce two types of results.
On the geometric side we observe relations between the index (or minimal Chern number) and the Betti numbers of a
compact symplectic manifold admitting a Hamiltonian action of a circle. In particular, the toric case, as well as that of maximal index, are analyzed, as described in Section \ref{geometry intro}.
An important feature of our approach is that the use of elliptic genera frees us from the usual positivity assumptions on the first Chern class (e.g.\ monotonicity or Fano hypotheses).

On the number theoretical side we use rigidity to obtain relations among products of Eisenstein series
which are dictated by the weights of the action on an almost complex manifold; these relations are, to the best of our knowledge, new (see Section \ref{number theory intro}).
Observe that, on the one hand, finding algebraic relations among modular forms is a non-trivial task. On the other hand, knowing which sets of integers can arise as the weights of a circle action is the content of the yet unresolved Smith problem. Yet somehow the rigidity of elliptic genera forms a bridge between these two mysterious phenomena.

\subsection{From number theory to geometry}\label{geometry intro}
Let $(\M,\J)$ be a compact almost complex manifold of dimension $2n$ with first Chern class given by $c_1$. We recall that the {\it index} of $(\M,\J)$
is the largest integer $\k0$ such that, modulo torsion, $c_1=\k0 \eta$ for some non-zero $\eta\in H^2(\M;\Z)$.
For symplectic manifolds, namely manifolds that can be endowed with a closed, non-degenerate two form $\omega$, the set of
almost complex structures compatible with $\omega$ is contractible, and hence one can define Chern classes of the tangent bundle.
We can then consider the first Chern class and the index of a compact symplectic manifold $(\M,\omega)$.

In analogy with algebraic geometry, it is natural to ask if there is a relation between the index and the
Betti numbers of $(\M,\omega)$. If we restrict to symplectic manifolds admitting a Hamiltonian circle action with isolated fixed points, then there has already been some progress in this direction.
Indeed, in \cite{S} the third author showed that for such manifolds the index, which coincides with the minimal Chern number, is bounded above by $n+1$.
In \cite{GHS} the authors proved that under the monotonicity assumption, namely $c_1=[\omega]$, and two additional technical hypotheses, there are several relations between
the index and the Betti numbers which mirror results obtained in algebraic geometry for Fano varieties; see also \cite{C} for recent developments.

In this paper we employ elliptic genera to derive relations between the index and the Betti numbers of a compact symplectic manifold admitting a Hamiltonian circle action with
isolated fixed points.
The use of elliptic genera of level $N$ is motivated by a result of Hirzebruch (see Theorem \ref{special values}) which asserts that the value of the elliptic genus at some special points
recovers information about the $\chi_y$-genus, which in turns depends on the Betti numbers of $\M$, and about the indices of some tensor powers of the line bundle $L=(\wedge^n  T^*\M)^{\frac{1}{\k0}}$ (see \cite{S}).

The first result that we obtain concerns compact symplectic manifolds admitting a larger action, namely that of a compact torus whose dimension
equals half the dimension of the manifold; these spaces are also known as  symplectic toric manifolds and the action as a toric action.
\begin{corollary}\label{Betti toric}
Let $(\M,\omega)$ be a compact symplectic manifold of dimension $2n$ that can be endowed with a toric action. Let $\k0$ be the index of $(\M,\omega)$, $b_j(\M)$ the $j$-th Betti numbers of $\M$ and $\mathbf{b}$ the vector
$(b_0(\M),b_2(\M),\ldots,b_{2n-2}(\M), b_{2n}(\M))$. Then
\begin{equation}\label{corollary toric}
\sum_{j=0}^{\k0-1}y^j \quad \text{divides} \quad \sum_{j=0}^n b_{2j}(\M)y^j.
\end{equation}
In particular we have $\k0\leq n+1$ and the following holds:
\begin{enumerate}[leftmargin=*]
            \item[\rm{(1)}] If $\k0 = n+1$, then $(\M,\omega)$ is symplectomorphic to $\C P^n$ with Fubini-Studi form suitably rescaled, and the symplectomorphism intertwines the torus
            action on $\M$ with the standard toric action on $\C P^n$.
            \item[\rm{(2)}] If $\k0 = n$, then $\mathbf{b}=(1,2,2,\dots,2,2,1)$.
            \item[\rm{(3)}] If $\k0 = n-1$, then we have for some non-negative  integer $m$
            \[
            \mathbf{b}=(1,\, 1+m, \, 2+m, \, 2+m, \, \dots \, ,\, 2+m, \, 2+m,\, 1+m, \, 1).
            \]

            \item[\rm{(4)}] If $\k0 = n-2$, then we have for some non-negative integer $m$
            \[
            \mathbf{b}=(1,\, 1+m, \, 1+2m,\,  2 + 2m, \, 2+2m, \, \dots \, , 2+2m,\,  1+2m,\, 1+ m,\, 1).
            \]

        \end{enumerate}
\end{corollary}
Hence $\C P^n$ is the only symplectic toric manifold of index $n+1$, regardless of the monotone assumption. As for $\k0=n$, it is known that if the symplectic toric manifold $(\M,\omega)$ is monotone, then
$n=2$ and $\M$ is $\C P^1\times \C P^1$ (see \cite[Corollary 5.12]{GHS}).

In Corollary \ref{reflexive h vector} we derive a straightforward translation of Corollary \ref{corollary toric} for a smooth reflexive polytope $\Delta$, where the vector of even Betti numbers is replaced by the $h$-vector of $\Delta$, and the index by the great common divisor of the affine lengths of its edges.

Next, we specialize to the case in which the index
is maximal. First of all observe that $\C P^n$, endowed with the Fubini-Studi symplectic form, is an example of a compact symplectic manifold with
a Hamiltonian action of a circle and isolated fixed points: this can be obtained from the standard toric action by restricting to a generic subcircle. Moreover its
index is exactly $n+1$ and its elliptic genus of level $n+1$ vanishes (see \cite{HBJ} and also Proposition \ref{toric vanishing} for an alternative proof).
In the following we prove that -- up to homotopy equivalence and complex cobordism -- the converse is also true.
\begin{thm}\label{main geometry}
Let $(\M,\omega)$ be a compact, connected symplectic manifold of dimension $2n$ which can be endowed with a Hamiltonian action of a circle with isolated fixed points.
 Assume that the index is maximal, i.e., $\k0=n+1$.
Then $\M$ is complex cobordant to $\C P^n$ if and only if its elliptic genus of level $n+1$ vanishes. Moreover,
if the elliptic genus of level $n+1$ vanishes, then $\M$ is homotopy equivalent to $\C P^n$.
\end{thm}
We would like to remark that, even if the hypothesis of the vanishing of the elliptic genus of level $n+1$ seems strong, it is automatically satisfied
if the type of the circle action is not zero modulo $n+1$ (see Definition \ref{balanced} and Theorem \ref{rigidity}); for instance this is the case for all symplectic toric manifolds (Proposition \ref{toric vanishing}).

\subsection{From geometry to number theory}\label{number theory intro}

We now discuss an application of the elliptic genus to number theory. Namely, we show, using the rigidity of the elliptic genus, that manifolds with circle actions yield many non-trivial relations between modular forms.

Let $k$ and $N$ be positive integers with $N \geq 2$. We consider the \emph{Eisenstein series} $G_{k,N}(\tau)$ defined for $\tau \in \H := \{\tau \in \C: \operatorname{Im}(\tau) > 0\}$ by the Fourier expansion
    \begin{align}\label{eisenstein expansion}
    G_{k,N}(\tau) :=  -\sum_{n=1}^{\infty}\left(\sum_{d \mid n}\left(\frac{n}{d} \right)^{k-1}\frac{\zeta_{N}^{-d} + (-1)^{k}\zeta_{N}^{d}}{(k-1)!}\right)e^{2\pi i n \tau} + \begin{dcases}
    \frac{1+\zeta_{N}}{2(1-\zeta_{N})} & \text{if }k = 1, \\
        \frac{B_{k}}{k!} & \text{if } k > 1,
        \end{dcases}
    \end{align}
    where $\zeta_{N} := e^{\frac{2\pi i}{N}}$ is a primitive $N$-th root of unity and $B_{k}$ denotes the $k$-th Bernoulli number. The Eisenstein series $G_{k,N}$ is a modular form of weight $k$ for the subgroup
    \[
    \Gamma_{1}(N): = \left\{\begin{pmatrix} a & b \\ c & d \end{pmatrix} \in \SL_{2}(\Z): c \equiv 0 \pmod N, a \equiv d \equiv 1 \pmod N\right\}
    \]
    of $\SL_{2}(\Z)$. This means that $G_{k,N}$ is holomorphic on $\H$ and at the cusps and satisfies the transformation law
    \[
    G_{k,N}\left(\frac{a\tau + b}{c\tau + d}\right) = (c\tau + d)^{k}G_{k,N}(\tau)
    \]
    for all $\left( \begin{smallmatrix} a & b \\ c & d \end{smallmatrix} \right) \in \Gamma_{1}(N)$ and $\tau \in \H$ (compare Section~\ref{section modular forms}). It is a fundamental fact that the vector space of all modular forms of fixed weight $k$ for $\Gamma_{1}(N)$ is finite-dimensional. This can be exploited to obtain relations between modular forms, for example. The following theorem roughly states that manifolds with circle actions yield relations for products of Eisenstein series.

\begin{thm}\label{main number theory}
    Let $(\M,\J)$ be a compact, connected, almost complex manifold of dimension $2n$ which is acted on effectively by a circle with a non-empty set of isolated fixed points. For a fixed point $P \in \M^{S^{1}}$ we let $w_{1}(P),\dots,w_{n}(P) \in \Z \setminus \{0\}$ denote the weights of the circle action at $P$. Let $\k0$ be the index of $(\M,\J)$ and suppose that $N$ divides $\k0$. Then, for $k > n$ we have the following relations of products of Eisenstein series
    \begin{equation}\label{product eisenstein}
    \sum_{I \in P_{n}(k)}\left(\sum_{ P \in \M^{S^{1}}}\frac{m_{I}\left(w_{1}(P),\dots,w_{n}(P)\right)}{w_{1}(P)\cdots w_{n}(P)}\right)G_{I,N}(\tau) = 0,
    \end{equation}
    where $P_{n}(k)$ is the set of all partitions of $k$ with at most $n$ parts, $m_{I}(x_{1},\dots,x_{n})$ denotes the monomial symmetric polynomial\footnote{We recall that given indeterminates $x_1,\ldots,x_n$ and a sequence of non-negative integers $I=(r_1,\ldots,r_n)$, the \emph{monomial symmetric polynomial} $m_I(x_1,\ldots,x_n)$
    is defined as the sum of all monomials $x^J$, where $J=(j_1,\ldots,j_n)$ ranges over all distinct permutations of $I$,
    and $x^I:=x_1^{r_1}\cdots x_n^{r_n}$.}, and $G_{I,N}(\tau) = G_{j_{1},N}(\tau)\cdots G_{j_{n},N}(\tau)$ for $I = [j_{1},\dots,j_{n}]$.
\end{thm}

The crucial idea of the proof of Theorem~\ref{main number theory} is
as follows: Since $\M$ is endowed with a circle action, we can
consider the equivariant elliptic genus $\varphi_{N}(\M,t)$ of level
$N$ associated to $\M$, which depends on an additional parameter $t
\in S^{1}$. The rigidity theorem (see Theorem~\ref{rigidity} below)
states that if $N$ divides the index $\k0$, then $\varphi_{N}(\M,t)$
is actually independent of $t$. In particular, if we consider the
Laurent expansion of $\varphi_{N}(\M,t)$ around $t = 1$, then all
coefficients apart from the constant term vanish identically. On the
other hand, the Laurent coefficients are essentially given by the
linear combinations of products of Eisenstein series from
Theorem~\ref{main number theory}. We refer the reader to
Section~\ref{section number theory proofs} for the details of the
proof.

\begin{rmk}
If it is known that the elliptic genus $\varphi_N(\M, t)$ vanishes
identically (see Theorem \ref{rigidity} and Proposition \ref{toric
vanishing} for cases in which this automatically happens), then in
the above theorem $k$ can be taken to be at least  $n$.
For $k<n$ equation \eqref{product eisenstein} does not give
meaningful relations, as a short argument involving the localization
formula in equivariant cohomology shows that the coefficient of
$G_{I,N}(\tau)$ is zero for every partition $I\in P_n(k)$.
\end{rmk}

We finish this section with an example illustrating the kinds of relations obtained from Theorem~\ref{main number theory}.

\begin{exm}
        Let $\M = \C P^{2}$ and $N = 3$. Then $\M$ is endowed with an $S^{1}$-action having three fixed points $P,Q,R$ with weights
        \[
        w_{1}(P) = x, \, w_{2}(P) = y, \qquad w_{1}(Q) = -x, \, w_{2}(Q) = y-x, \qquad w_{1}(R) = -y,\, w_{2}(R) = x-y.
        \]
        Here $x,y \in \Z$ can be chosen arbitrarily as long as no weight equals zero. We obtain from Theorem~\ref{main number theory} the relations
        \[
        \sum_{I \in P_{2}(k)}\left(\frac{m_{I}(x,y)}{xy} + \frac{m_{I}(-x,y-x)}{(-x)(y-x)} + \frac{m_{I}(-y,x-y)}{(-y)(x-y)}\right)G_{I,3}(\tau) = 0
        \]
        for all $k > 2$. The expression in the big brackets can be evaluated using a computer algebra system. Explicitly, we get the following relations:
        \begin{align*}
         4G_{1,3}G_{3,3}+G_{2,3}^{2} + 5G_{4,3} &= 0,\\
         - G_{2,3}G_{3,3}+ G_{5,3}  &= 0, \\
         4G_{1,3}G_{5,3} + 2G_{2,3}G_{4,3}+G_{3,3}^{2}+7G_{6,3} &= 0 ,\\
         -G_{2,3}G_{5,3}-G_{3,3}G_{4,3} + 2G_{7,3} &= 0.
        \end{align*}
        We have checked these identities using the Fourier expansions of the Eisenstein series $G_{k,3}$.
        We also refer the reader to Proposition~\ref{proposition general relation CPn} for more general relations between products of Eisenstein series coming from the rigidity of the elliptic genus of $\C P^{n}$.
    \end{exm}
        \subsection{Outline of the paper}
    As this paper is aimed to be for a general audience, in Section \ref{background} we write an extensive introduction to the subject and recall all of the necessary concepts.
    In particular, in Subsection \ref{generalities s1 action} some generalities about $S^1$-actions on almost complex and symplectic manifolds are recalled; in Subsection \ref{eckt} some facts
    about the equivariant cohomology and K-theory ring are given, especially the so-called localization formulas in both settings.  Subsection \ref{section modular forms} contains a short introduction to modular forms and the definition of Eisenstein series. In Subsection \ref{genera} we recall what genera associated to a power series are, and finally in Subsection \ref{def elliptic genus} we define what the elliptic genus of level $N$ is. In Subsection \ref{value cusps} we recall how the elliptic genus recovers important information about the $\chi_y$-genus and the indices of some tensor powers of the line bundle $(\wedge^n T^*\M)^{\frac{1}{N}}$. The interpretation of the elliptic genus as the topological index of an infinite tensor product is in Subsection \ref{vvb}.
 We define the type of the action, as well as what it means for an action to be $N$-balanced, in Subsection \ref{type of action}. We also give an alternative proof of Proposition \ref{balanced index},
 already known in the literature, which identifies for which integers $N$ an action is $N$-balanced.  Section \ref{section rigidity} is devoted to describing what rigidity of the elliptic genus means, and the rigidity theorem is recalled here.
 
 In Section \ref{number theory to geometry} we use the rigidity of the elliptic genus to derive topological results about compact symplectic manifolds endowed with a Hamiltonian circle action
 with isolated fixed points. In particular, in Subsection \ref{number theory to geometry I} we first specialize to toric actions and prove Proposition \ref{toric vanishing} which asserts that, for every positive integer $N$ dividing the index of
 the manifold, the elliptic genus of level $N$ vanishes identically. Proposition \ref{divisibility chiy genus} gives a divisibility criterion for the $\chi_y$-genus which, together with Proposition \ref{toric vanishing}, are the key ingredients for the proof of Corollary \ref{Betti toric}, which is given on page \pageref{proof corollary}. In subsubsection \ref{section reflexive} we recall some generalities
 about reflexive polytopes and translate Corollary \ref{Betti toric} in this setting: this is the content of Corollary \ref{reflexive h vector}.
 In Subsection \ref{closer look} we specialize to the case in which the index is maximal. Before doing so we introduce some polynomials, see equation \eqref{Hmk}, which generalize the so-called Hilbert polynomial and play a key role in the proof of Theorem \ref{main geometry}. We prove some of their symmetries (Proposition \ref{symmetries H}) and compute them, as an example, for the complex projective space (see Proposition \ref{projective space H}). In Proposition \ref{fixed points and polynomials} we show what the relation between these new polynomials and the
 number of fixed points is. Theorem \ref{main 2} is the key ingredient for the proof of Theorem \ref{main geometry} and implies that, for a compact symplectic manifold of dimension $2n$ acted on by a circle in a Hamiltonian way and with isolated fixed points, having index $\k0=n+1$ and vanishing elliptic genus of level $n+1$ implies the number of fixed points, and hence the Euler characteristic, to be $n+1$. To conclude the proof of Theorem \ref{main geometry}, which is given on page \pageref{proof main geometry}, we then need Theorem \ref{main 3}, whose proof combines results of Hattori \cite{Ha}, Tolman \cite{T} and Charton \cite{C}.

In Section \ref{section number theory proofs} we give a proof of Theorem \ref{main number theory}. Then in Subsection \ref{section coadjoint orbit}, after recalling some standard
facts, we give a formula for computing the coefficients
appearing in Theorem \ref{main number theory} whenever $\M$ is a coadjoint orbit (Proposition \ref{formula divided}). Finally, in Subsection \ref{explicit}, we compute the relations among Eisenstein series given by Theorem \ref{main number theory} when $\M$ is the complex projective space.
%%%%%%
\section{Background}\label{background}
\subsection{The \texorpdfstring{$S^1$}{Lg}-actions on almost complex and symplectic manifolds}\label{generalities s1 action}
In this subsection we recall some standard facts about circle actions on almost complex and symplectic manifolds.

Let $\ac$ be an almost complex manifold of dimension $2n$ endowed with a circle action that preserves $\J$. This means that the endomorphism $\J\colon T\M \to T\M$
is equivariant with respect to the action of $S^1$ induced on the tangent bundle $T\M$.
Henceforth we assume that the $S^1$-action on $\ac$ has fixed points, and denote the set of fixed points by $\fix$.  Similarly, for every subgroup $\Z_k$ of $S^1$,
we denote the set of points whose stabilizer is $\Z_k$ by $\M^{\Z_k}$.
We recall that, given a fixed point $P\in \M^{S^1}$, there exist complex coordinates $z_1,\ldots,z_n$ on $T\M|_P\simeq \C^n$
and integers $w_{1}(P),\ldots,w_{n}(P)$ such that the $S^1$-action on $T\M|_P$ is given by
$$
S^1 \ni \lambda \cdot (z_1,\ldots,z_n) = \left(\lambda^{w_{1}(P)}z_1,\ldots,\lambda^{w_{n}(P)}z_n\right).
$$
Such integers are called the \emph{weights of the $S^1$-action} at the fixed point $P$. Note that $P\in \fix$ is an isolated fixed point if and only if none of its weights is zero; this is a consequence of the fact that, after choosing an $S^1$-invariant metric on $\M$, in a neighborhood $U\subset \M$ of $P$ the exponential map with respect to this metric intertwines the $S^1$-action on $T\M|_P$ with that on $U$.

If $\M$ carries additional structure, then we require that such structure is preserved by the circle action. For instance, let $(\M,\omega)$ be a symplectic manifold, which in this article is
always assumed to be compact and connected. If $S^1$ acts on it, then we require the $S^1$-family of diffeomorphisms to be indeed symplectomorphisms.
This translates into the following formula: let $\xi$ be a vector in $\mathrm{Lie}(S^1)$ and $\xi^\#$ the corresponding vector field on $\M$. Then the flow of diffeomorphisms associated to $\xi^\#$
is a flow of symplectomorphisms if and only if
$$
d\left(\iota_{\xi^\#}\omega\right) = 0\,.
$$
In the case in which the closed form above is exact, namely if there exists $\psi\colon \M\to \R$ such that $\iota_{\xi^\#}\omega = d\,\psi$, then the $S^1$-action is called \emph{Hamiltonian}
and the function $\psi$ the \emph{moment map} of the action.
If there is a whole compact torus $\T$ acting on $(\M,\omega)$ via symplectomorphisms, then the notion of Hamiltonian action generalizes to this case in the
following way.
\begin{defin}\label{hamiltonian t action}
Let $(\M,\omega)$ be a symplectic manifold and let $\T$ be a compact torus acting on it via symplectomorphisms with Lie algebra $\mathfrak{t}$. Then the $\T$-action is called \emph{Hamiltonian} if there exists a map $\psi\colon \M\to \mathfrak{t}^*$ such that the following conditions hold:
\\
\noindent $\bullet $ $\psi$ is $\T$ invariant;\\
$\bullet$ for every
$\xi \in \mathfrak{t}$ the following identity holds:
\begin{equation}\label{moment map}
\iota_{\xi^\#}\omega = d \psi^\xi\,,
\end{equation}
\\
\noindent
where $\xi^\#$ is the vector field corresponding to the Lie algebra element $\xi$ and the function $\psi^\xi$ is given by $\psi^\xi(P):=\langle \psi(P),\xi \rangle$ (here $\langle \cdot , \cdot \rangle$ denotes the evaluation between $\mathfrak{t}^*$ and $\mathfrak{t}$).
\end{defin}
Equation \eqref{moment map} implies that, if $\M$ is compact, then a Hamiltonian action always has fixed points. A lower bound can be found as follows.
Suppose first that the torus is one-dimensional. If the fixed points are isolated, then the moment map is a Morse function with only even indices, and a Morse theory argument implies that
\begin{equation*}\label{Nibi}
b_{2j}(\M)=N_j\,,\quad\text{for all}\;\;j\in\{0,\ldots,n\},
\end{equation*}
where $b_{2j}(\M)$ is the $2j$-th Betti number of $\M$ and $N_j$ is the number of fixed points with $j$ negative weights.
Since for a compact symplectic manifold $\M$ of dimension $2n$ the non-degeneracy of the symplectic form implies $b_{2j}(\M)\neq 0$ for all $j\in\{0,\ldots,n\}$, we conclude that
for a Hamiltonian circle action $|\M^{S^1}|\geq n+1$. The same conclusion holds for actions of tori of higher dimensions: it suffices to restrict the action to a circle subgroup.

Since the vector space tangent to the torus orbits of a Hamiltonian action is isotropic, the largest dimension of a compact torus that acts
on a $2n$-dimensional symplectic manifold is exactly $n$.
\begin{defin}\label{symplectic toric}
A \emph{symplectic toric manifold} is a compact symplectic manifold of dimension $2n$ endowed with the effective, Hamiltonian action of a compact torus $\T$ of dimension $n$.
We denote this space with the triple $(\M,\omega,\psi)$, where $\psi\colon \M \to \mathfrak{t}^*$ and $\mathfrak{t}=\mathrm{Lie}(\T)$.
\end{defin}
By the celebrated Convexity Theorem of Atiyah \cite{A} and Guillemin-Sternberg \cite{GSt}, the image of the moment map is a convex polytope. However, for symplectic toric manifolds this
polytope has very special features, and is called a {\it smooth} (or {\it Delzant}) polytope. These are defined combinatorially as follows.

\begin{defin}\label{delzant}
Let $\Delta\subset \R^n$ be an $n$-dimensional polytope. Then $\Delta$ is called {\it smooth} (or {\it Delzant}) if the following three properties are satisfied:
\begin{itemize}
\item[(D1)] $\Delta$ is \emph{simple}: there are exactly $n$ edges meeting at each vertex;
\item[(D2)] every vertex is \emph{rational}: the lines supporting the $n$ edges meeting at a vertex $v$ are given by $v+ t\cdot w_j$, with $t\in \R$ and $w_j\in \Z^n$ for every $j\in\{1,\ldots,n\}$, for every vertex $v$;
\item[(D3)] every vertex is \emph{smooth}: for each vertex $v$ the vectors $w_1,\ldots,w_n$ above can be chosen to be a $\Z$-basis of $\Z^n$.
\end{itemize}
\end{defin}
Given the Convexity Theorem, it is not hard to see that the image of the moment map of a $2n$-dimensional symplectic toric manifold is a smooth $n$-dimensional polytope. However
a celebrated theorem of Delzant \cite{D} asserts that the polytope characterizes completely the symplectic toric manifold up to equivariant symplectomorphisms. Moreover for any smooth polytope $\Delta$ one can find a unique (modulo equivariant symplectomorphism) symplectic toric manifold whose image of the moment map is exactly $\Delta$.

\subsection{Equivariant cohomology, characteristic classes, and \texorpdfstring{$K$}{Lg}-theory}\label{eckt}
For a more detailed exposition and for proofs of the following facts see for instance \cite{AB, BV,GS2} for equivariant cohomology and characteristic classes and \cite{A2,AS1,AS2,AS3} for
equivariant $K$-theory.
We recall that, given a manifold $\M$ acted on continuously by a circle $S^1$, according to the Borel model the \emph{equivariant cohomology ring} $H^*_{S^1}(\M;R)$ of $\M$ with coefficients in the ring $R$
 is defined to be the ordinary cohomology ring of the orbit space $\M \times_{S^1} E S^1$, where $E S^1$ is a contractible space on which $S^1$ acts freely, and the $S^1$-action on
$\M\times E S^1$ is the diagonal one. The space $E S^1$ can, for instance,  be chosen to be the unit sphere
$S^\infty$ in $\C^\infty$; however $E S^1$ is unique up to homotopy equivalence.
For instance we have that
\begin{equation}\label{restriction p}
H^*_{S^1}(pt; R)=H^*\left(S^\infty/S^1;R\right)=H^*\left(\C P^\infty;R\right)=R[x],
\end{equation}
where $R$ is the coefficient ring and $x$ has degree two.
Similarly, if we have a compact torus $\T$ of dimension $m$, then $H^*_\T(\M;R)$ is defined to be $H^*(\M\times_\T E \T)$, where $E \T$ can be chosen to be $(S^\infty)^m$. Hence
$H^*_\T(pt;R)=R[x_1,\ldots,x_m]$, where each $x_j$ has degree two.
If the ring $R$ is a field, then $H^*_\T(pt;R)$ is identified with the symmetric algebra on $\mathfrak{t}^*$, where $\mathfrak{t}$ is the Lie algebra of $\T$, and is denoted
by $\mathbb{S}(\mathfrak{t}^*)$.

Any $S^{1}$-equivariant map between two spaces induces a pull-back map in equivariant cohomology. For instance, let $\M^{S^{1}}$ denote the fixed point set, which we are assuming to be non-empty. Then the ($S^{1}$-equivariant) inclusion $\iota\colon \M^{S^{1}} \hookrightarrow \M$ induces a map
$\iota^*\colon H_{S^{1}} ^*(\M;R)\to H_{S^{1}}^*(\M^{S^{1}};R)$ which is well-understood in many cases; for example,
if $\M$ can be endowed with a Hamiltonian torus action, then Kirwan \cite{K} proved that $\iota^*$ is injective.
Note that $H_{S^{1}}^*(\M^{S^{1}};R)$ is an easier object to handle, and if in particular the set of
fixed points consists of finitely many elements, then $H_{S^{1}}^*(\M^{S^{1}};R)=\oplus_{P\in \M^{S^{1}}} R[x]$. Henceforth the restriction of an equivariant cohomology class $\alpha$
to a fixed point $P$ is denoted by $\alpha(P)$.

The equivariant cohomology ring often recovers properties of the non-equivariant one. Indeed, one can consider the inclusion map $\{e\}\hookrightarrow S^{1}$, where
$e$ denotes the identity element in $S^{1}$, observe that $H_{\{e\}}^*(\M;R)$ is just the ordinary cohomology ring, and study the corresponding pull-back
\begin{equation*}\label{restriction r}
r\colon H_{S^{1}}^*(\M;R)\to H^*(\M;R).
\end{equation*}
Often this map is well-understood too; for instance, in the Hamiltonian case, it is surjective \cite{K}.
Observe that if $\M$ is a point, then $r$ is given by evaluation at $x=0$.

Just as for the ordinary cohomology ring, there is a push-forward map in equivariant cohomology, also called an ``integration map''
\begin{equation}\label{integration}
\int_\M \colon H_{S^{1}}^*(\M;R)\to H_{S^{1}}^{*-\dim(\M)}(pt;R)
\end{equation}
which comes from integration along the fibers of the projection onto the second factor $\M\times_{S^{1}} E S^{1} \to \C P^\infty$.
Observe that the integral of an equivariant cohomology class of degree higher than the dimension of $\M$ is not necessarily zero.
In order to compute \eqref{integration} there is a very useful formula due to Atiyah-Bott \cite{AB} and Berline Vergne \cite{BV}, which is referred to
as the localization formula (in equivariant cohomology), and is as follows. For the sake of brevity we only describe it if $\M$
is almost complex and the fixed points are isolated; however it holds more generally for oriented manifolds and any type of fixed point set.
\begin{lemma}
Let $\ac$ be a compact almost complex manifold of dimension $2n$ acted on by a circle with discrete fixed point set $\M^{S^1}$. Let $\alpha$ be an equivariant cohomology class. Then we have
\begin{equation}\label{ABBV}
\int_\M \alpha = \sum_{P\in \M^{S^1}}\frac{1}{x^n}\frac{\alpha(P)}{w_1(P)\cdots w_n(P)},
\end{equation}
where $x$ is as in \eqref{restriction p}.
\end{lemma}
Consider now a complex vector bundle $V \to \M$ which is endowed with a compatible $S^1$-action, namely an action that makes
the projection map equivariant.
Then its \emph{equivariant Chern classes} are defined as the ordinary Chern classes,
of the associated bundle $V \times_{S^1} ES^1 \to \M \times_{S^1} ES^1$.
If $V$ is the tangent bundle $T\M$, then, using functoriality of Chern classes it is not hard to
prove that, given a fixed point $P$ with weights $w_1(P),\ldots, w_n(P)$, the restriction of the $j$-th equivariant Chern class $c_j^{S^1}$ of the tangent bundle $T\M$ to $P$ is given by
\begin{equation}\label{restriction chern}
c_j^{S^1}(P)=e_j(w_1(P),\ldots,w_n(P))x^j\,,
\end{equation}
where $e_j(x_1,\ldots,x_n)$ denotes the $j$-th elementary symmetric polynomial in $x_1,\ldots,x_n$.

Since the restrictions of equivariant Chern classes to ordinary cohomology are the ordinary Chern classes, formula \eqref{ABBV} can be very useful to
compute Chern numbers of $\M$ too. Let $c_1,\ldots,c_n$ denote the Chern classes of the tangent bundle of $(\M,\J)$, where $c_j\in H^{2j}(\M;\Z)$,
and $P(n)$ the set of partitions $\lambda = [\lambda_{1},\dots,\lambda_{k}]$ of $n$, i.e., $\lambda_{1} \geq \dots \geq \lambda_{k} > 0$ with $\lambda_{1} + \ldots + \lambda_{k} = n$. For each  $\lambda\in P(n)$ the \emph{Chern number} of $(\M,\J)$ associated to $\lambda$ is given by
    \begin{equation}\label{chern number def}
    C_{\lambda}(\M) = \int_{\M}c_{\lambda_{1}}\cdots c_{\lambda_{k}} \in \Z.
    \end{equation}
  Using formula \eqref{ABBV} it is easy to show that
  \begin{equation}\label{chern number weights}
  C_\lambda(\M)=\sum_{P\in \M^{S^1}}\frac{e_{\lambda_1}(w_1(P),\ldots,w_n(P))\cdots e_{\lambda_k}(w_1(P),\ldots,w_n(P))}{w_1(P)\cdots w_n(P)}\,.
  \end{equation}

The $K$-theory (resp. equivariant $K$-theory) ring $S^1$,  is given a manifold $\M$ acted on by a circle, the abelian group associated
to the semigroup of isomorphism classes of complex vector bundles (resp. complex vector bundles endowed with a compatible $S^1$-action), endowed with the operations
of direct sum and tensor product. It is denoted by $K(\M)$ (resp. $K_{S^1}(\M)$); for instance $K(pt)\simeq \Z$ and $K_{S^1}(pt)\simeq \Z[t,t^{-1}]$.
The inclusion $\{e\}\hookrightarrow S^1$ induces a map $K_{S^1}(\M)\to K(\M)$ which, in the case in which $\M$ is a point, is the evaluation at $t=1$.
The topological index (resp. equivariant index) of a bundle $V$ --regarded as an element of $K(\M)$-- is a push forward map
$
\ind\colon K(\M) \to K(pt)\simeq \Z
$
(resp. $\ind_{S^1}\colon K_{S^1}(\M)\to K_{S^1}(pt)\simeq \Z[t,t^{-1}]$), as defined in \cite{AS2}. Its computation can be carried out using the Atiyah-Singer formula \cite{AS3}
$$
\ind(V)=\int_\M \text{Ch}(V) \text{Todd}(\M),
$$
where $\text{Ch}\colon K(V)\to H^*(\M)$ denotes the Chern character homomorphism and $\text{Todd}(\M)$ the total Todd class of $\M$.
 As for the equivariant index there is, in analogy with equivariant cohomology, a localization
formula that allows to compute it, as proved by Atiyah and Segal in \cite{AS1}. As above, for simplicity we assume here that $\M$ is almost complex
and the fixed point set is discrete.
\begin{lemma}\label{localization K}
Let $\ac$ be a compact almost complex manifold of dimension $2n$ acted on by a circle with isolated fixed points.
Given an equivariant bundle $V\in K_{S^1}(\M)$, its equivariant index is given by
\begin{equation}\label{index K}
\ind_{S^1}(V)=\sum_{P\in \M^{S^1}} \frac{V(P)}{\prod_{j=1}^n \left(1-t^{-w_j(P)}\right)}\in \Z\left[t,t^{-1}\right]\,.
\end{equation}
\end{lemma}
Just for equivariant cohomology classes, given a bundle $V$ over $\M$ endowed with an $S^1$-action compatible with the projection, to compute its non-equivariant index
one can use formula \eqref{index K}, namely
\begin{equation}\label{index}
\ind(V)=\lim_{t\to 1}\sum_{P\in \M^{S^1}} \frac{V(P)}{\prod_{j=1}^n \left(1-t^{-w_j(P)}\right)}\,.
\end{equation}

\subsection{Modular forms for \texorpdfstring{$\Gamma_{1}(N)$}{Lg}}\label{section modular forms}
    We recall some basic facts about modular forms for the congruence subgroup $\Gamma_{1}(N)$, roughly following Appendix I of \cite{HBJ}. Note that the index of $\Gamma_{1}(N)$ in $\SL_{2}(\Z)$ is finite.
    The group $\SL_{2}(\Z)$ acts on the upper half-plane $\H$ by fractional linear transformations
    \[
    M\tau := \tfrac{a\tau + b}{c\tau +d}
    \]
    for $M = \left(\begin{smallmatrix}a & b \\ c & d \end{smallmatrix} \right) \in \SL_{2}(\Z)$, and this action extends to $\Q \cup \{\infty\}$ if we set $M \infty := \frac{a}{c}$ with $\frac{\pm 1}{0} = \infty$. The classes $\Gamma_{1}(N) \backslash (\Q\cup\{\infty\})$ are called the \emph{cusps} of $\Gamma_{1}(N)$. Since $\Gamma_{1}(N)$ has finite index in $\SL_{2}(\Z)$, it also has finitely many cusps.

    The geometric significance of the cusps is the fact that the quotient $\Gamma_{1}(N)\backslash \H$ is a non-compact Riemann surface which can be compactified by adding the finitely many cusps $\Gamma_{1}(N) \backslash (\Q\cup\{\infty\})$.

    A \emph{modular form} of weight $k \in \Z$ for $\Gamma_{1}(N)$ is a holomorphic function $f: \H \to \C$ which transforms as
    \[
    f(M\tau) = (c\tau+d)^{k}f(\tau)
    \]
    for all $M = \left(\begin{smallmatrix}a & b \\ c & d \end{smallmatrix} \right) \in \Gamma_1(N)$ and which is \emph{holomorphic at the cusps}, which means that for every $M \in \SL_{2}(\Z)$ the function $f|_{k}M(\tau) := (c\tau+d)^{-k}f(M\tau)$ has a \emph{Fourier expansion} of the form
    \begin{align}\label{eq fourier expansion cusp}
    f|_{k}M (\tau)= \sum_{\substack{n \in \Q_{0}^+ }}c_{M}(n)q^{n}
    \end{align}
    with coefficients $c_{M}(n) \in \C$ and $q := e^{2\pi i \tau}$. If $M\infty = \frac{a}{c}$, then we call the constant term $c_{M}(0)$ of $f|_{k}M$ the \emph{value of $f$ at the cusp $\frac{a}{c}$} and denote it by $f(\frac{a}{c})$. It is independent of the particular choice of $M$ with $M\infty = \frac{a}{c}$, and it can be computed as the limit
    \[
    f\left(\tfrac{a}{c}\right)= \lim_{\tau \to i\infty}f|_{k}M(\tau),
    \]
    or, equivalently, by taking the limit of $f|_{k}M$ as $q \to 0$ in the Fourier expansion \eqref{eq fourier expansion cusp}.

    We let $M_{k}(\Gamma_{1}(N))$ be the complex vector space of all modular forms of weight $k$ for $\Gamma_{1}(N)$. It is well-known that $M_{k}(\Gamma_{1}(N)) = \{0\}$ for $k \leq 0$, $M_{0}(\Gamma_{1}(N)) = \C$, and that $M_{k}(\Gamma_{1}(N))$ is finite-dimensional for all $k \in \N$.
    Furthermore, we let
    \[
    M(\Gamma_{1}(N)) := \bigoplus_{k = 0}^{\infty}M_{k}(\Gamma_{1}(N))
    \]
be the graded ring of all modular forms for $\Gamma_{1}(N)$.

    There are many ways to construct interesting examples of modular forms for $\Gamma_{1}(N)$. Here we consider the \emph{Eisenstein series}
    \begin{align}\label{definition eisenstein}
    G_{k,N}(\tau) := -\frac{1}{(2\pi i)^{k}}\sum_{(m,n) \in \Z^{2}\setminus \{(0,0)\}}\frac{\zeta_{N}^{m}}{(m\tau + n)^{k}} \in M_{k}(\Gamma_{1}(N))
    \end{align}
    for $N \geq 2$ and $k \geq 1$.  Actually, the defining series only converge if $k \geq 3$, but one can also make sense of the series for $k \in \{1,2\}$ (using the so-called Hecke trick) to obtain Eisenstein series of weight one and two. The Fourier expansion of $G_{k,N}$ at $\infty$ is given by \eqref{eisenstein expansion}.

    \subsection{Genera associated to formal power series}\label{genera}
We first recall that a \emph{stably almost complex manifold} (also known as \emph{unitary manifold}) is a manifold $\M$ together with a complex structure on the stable tangent bundle
(namely a complex structure on $T\M \oplus (\M\times \R^k)$ for some $k\geq 1$, where $\M\times \R^k$ denotes the trivial real bundle of rank $k$). The \emph{complex cobordism ring}
$\Omega$ is the ring of cobordism classes of stable complex manifolds. A classical result of Milnor \cite{M} and Novikov \cite{N} (see also Stong \cite{Stong}) that is used below asserts that two almost complex manifolds are stably cobordant if and only if their Chern numbers are the same.

    A \emph{complex multiplicative genus} (or just a \emph{genus}) is a ring homomorphism $\varphi: \Omega \otimes \Q \to R$, where $R$ is a given ring; for example, $R$ could be $\Z$ or $\Q$, the polynomial ring $\Q[y]$, or the ring $M(\Gamma_{1}(N))$ of modular forms for $\Gamma_{1}(N)$.

    One can construct genera from normalized formal power series, as we explain now. Let
    \[
    Q(x) = 1 + a_{1}x + a_{2}x^{2} + a_{3}x^{3}+\dots
    \]
    be a {normalized formal power series} with coefficients $a_{k} \in R$. Let $(\M,\J)$ be an almost complex compact manifold of dimension $2n$. The following procedure yields the genus $\varphi_{Q}(\M)$ associated to $Q$:
\\$\;$\\
    \emph{Step 1}: We introduce variables $x_{1},\dots,x_{n}$ to which we assign weight one. Then, for example, $x_{1}x_{2}^{2}$ has weight three. Consider the product
    \begin{align*}
    Q(x_{1})Q(x_{2})\cdots Q(x_{n})
    &= 1 + \sum_{k=1}^{\infty}p_{k}(x_{1},\dots,x_{n}),
    \end{align*}
    where the polynomial $p_{k}$ consists of all summands of weight $k$. The first three polynomials $p_{k}$ are given by
    \begin{align*}
    p_{1} &= a_{1}\left(x_{1} + \dots + x_{n}\right), \\
    p_{2} &= a_{1}^{2}\left(x_{1}x_{2} + x_{1}x_{3} + \dots + x_{n-1}x_{n}\right) + a_{2}\left(x_{1}^{2}+\dots + x_{n}^{2}\right), \\
    p_{3} &= a_{1}^{3}\left(x_{1}x_{2}x_{3} + x_{1}x_{2}x_{4} + \dots\right) + a_{1}a_{2}\left(x_{1}x_{2}^{2} + x_{2}x_{1}^{2} + \dots\right) + a_{3}\left(x_{1}^{3}+\dots+x_{n}^{3}\right).
    \end{align*}
    Note that $p_{k}$ is homogeneous of degree $k$.
\\$\;$\\
    \emph{Step 2}: Each polynomial $p_{k}$ is symmetric, so it can be written as a polynomial in the elementary symmetric polynomials\footnote{Recall that the {\it elementary symmetric polynomials} are defined by $\sigma_{\ell}(x_{1},\dots,x_{n}) \hspace{-17pt}:= \sum_{1 \leq j_{1}< \dots < j_{\ell} \leq n}x_{j_{1}}\cdots x_{j_{\ell}}$. For example, $\sigma_{0} (x_1 \dots, x_n)= 1, \sigma_{1} (x_1 \dots, x_n)= x_{1}+\dots+x_{n},$ and $\sigma_{n} = x_{1}\cdots x_{n}$.} $\sigma_{\ell}$ with $\ell \leq k$, i.e., there is a polynomial $Q_{k}$ in $k$ variables such that
    \[
    p_{k}(x_{1},\dots,x_{n}) = Q_{k}(\sigma_{1}(x_{1},\dots,x_{n}),\dots,\sigma_{k}(x_{1},\dots,x_{n})).
    \]
    Again, we list some explicit formulas for the first few $q_{k}$:
    \begin{align*}
    Q_{1}(y_1) &= a_{1}y_{1}, \\
    Q_{2}(y_1,y_2) &= a_{2}y_{1}^{2}+\left(a_{1}^{2}-2a_{2}\right)y_{2}, \\
    Q_{3}(y_1,y_2,y_3) &= a_{3}y_{1}^{3}+(a_{1}a_{2}-3a_{3})y_{1}y_{2} + \left(a_{1}^{3}+3a_{3}-3a_{1}a_{2}\right)y_{3}.
    \end{align*}
    If we think of $y_{\ell}$ as having weight $\ell$, then the polynomial $Q_{\ell}$ consists only of terms of weight $\ell$. This implies that if the coefficients $a_{k}$ lie in some graded ring $R$ with $a_{k} \in R_{k}$, then the coefficients of $Q_{k}$ lie in $R_{k}$, as well.
\\$\;$\\
    \emph{Step 3}: Since the manifold $\M$ is almost complex, i.e., its tangent bundle has a complex structure, its Chern classes $c_{1}(\M),\dots,c_{n}(\M)$ (which are by definition the Chern classes of the tangent bundle of $\M$) are defined. We plug in $c_{\ell}(\M)$ for $\sigma_{\ell}$ in $Q_{n}$ and integrate over $\M$ to obtain an element of $R$ (even $R_{n}$ if $R$ is graded), which we call the \emph{genus $\varphi_{Q}(\M)$ associated to $Q$}, i.e.,
    \[
    \varphi_{Q}(\M) = \int_{\M}Q_{n}(c_{1}(\M),\dots,c_{n}(\M)).
    \]
\noindent The above procedure is tedious to do by hand, but can easily be done by a computer algebra system.

    \begin{rmk}
        It is sometimes convenient to allow non-normalized power series $Q(x) = a_{0}+a_{1}x+a_{2}x^{2} + \dots$ with $a_{0} \neq 0$. We can go through the same procedure as above and still get a genus associated to $Q$. If $Q$ is not normalized, then the normalized
         series $a_{0}^{-1}Q(a_{0}x)$ yields the same genus as $Q$.
    \end{rmk}

    \begin{rmk}
     We can write $\varphi_{Q}(\M)$ as a linear combination of Chern numbers (see \eqref{chern number def})
    \[
    \varphi_{Q}(\M) = \sum_{\lambda \in P(n)}f_{\lambda}C_{\lambda}(\M)
    \]
    for some $f_{\lambda} \in R$ (even in $R_{n}$ if $R$ is graded). Note that $f_{\lambda}$ is a polynomial in the $a_{k}$'s, which depends on $\lambda$ (hence on $n$), but not on the manifold $\M$. Again, the coefficients $f_{\lambda}$ can easily be computed by a computer algebra system.
    \end{rmk}

    \begin{exm}
        Let $n = 2$. We write
        \[
        Q(x_{1})Q(x_{2}) = 1 + P_{1}(x_{1},x_{2}) +P_{2}(x_{1},x_{2})+ \dots,
        \]
        where
        \begin{equation*}
        P_1(x_1,x_2):=a_1 \left(x_1 +x_2\right), \quad P_2(x_1,x_2):= a_2 \left( x_1^2 +x_2^2\right)+a_1^2x_1x_2.
        \end{equation*}
        In terms of the elementary symmetric polynomials, we have
        \begin{align*}
        x_{1}^{2} + x_{2}^{2} = (x_{1}+x_{2})^{2} - 2x_{1}x_{2} = \sigma_{1}^{2} - 2\sigma_{2}, \qquad x_{1}x_{2} = \sigma_{2},
        \end{align*}
        and hence
        \[
        P_{2}(x_{1},x_{2}) = a_{2}\left(\sigma_{1}^{2}-2\sigma_{2}\right) + a_{1}^{2}\sigma_{2} = a_{2}\sigma_{1}^{2}+\left(a_{1}^{2}-2a_{2}\right)\sigma_{2}= Q_{2}(\sigma_{1},\sigma_{2}).
        \]
        Plugging in the Chern numbers and integrating over $\M$, we obtain
        \[
        \varphi_{Q}(\M) = a_{2}C_{[1,1]}(\M) + \left(a_{1}^{2}-2a_{2}\right)C_{[2]}(\M).
        \]
    \end{exm}

    \begin{exm}
        The \emph{Hirzebruch $\chi_{y}$-genus} is the $\Q[y]$-valued genus associated to the power series
        \[
        \frac{x}{1-e^{-x}}(1+ye^{-x}).
        \]
        If we plug in $\frac{x}{e^{x}-1} = \sum_{n=0}^{\infty}B_{n}\frac{x^{n}}{n!}$ with the Bernoulli numbers $B_{n} \in \Q$ ($B_{0} =1 , B_{1} = -\tfrac{1}{2}, B_{2} = \tfrac 16$), then we find the explicit expansion
        \[
        \frac{x}{1-e^{-x}}(1+ye^{-x}) = a_0(y)+ \sum_{k=1}^{\infty} a_{k}(y)x^{k},
        \]
        where
        \[
        a_0(y):=(1+y), \qquad a_k(y):= \sum_{n=0}^{k}\frac{B_{n}}{n!(k-n)!}+y\frac{B_{k}}{k!}.
        \]
        In particular, every coefficient $a_{k}$ is a polynomial in $y$ of degree $\leq 1$ with rational coefficients. Note that the power series is not normalized. The associated genus $\chi_{y}(\M)$ for a manifold of dimension $2n$ is a linear combination of products of the form $a_{0}^{n-\ell}a_{k_{1}}\cdots a_{k_{\ell}}$. Hence $\chi_{y}(\M)$ is a polynomial in $y$ of degree $n$ with rational coefficients. Indeed more is true: the coefficient of $y^p$ is always an integer, for all $p\in\{0,\ldots,n\}$, as it is the topological index of the bundle $\wedge^p T$ (see \cite[page 61]{HBJ}).

        Now let $n = 2$ and consider $\M = \C P^{2}$. Its Chern numbers are given by
        \begin{align*}
        C_{[1,1]}\left(\C P^{2}\right) &= 9,  \qquad    C_{[2]}\left(\C P^{2}\right) = 3.
        \end{align*}
        The relevant polynomials are
        \begin{align*}
        a_{0}(y) &= 1+y, \qquad
        a_{1} (y)= B_{0}+B_{1}+yB_{1} = \frac{1}{2}-\frac{y}{2},\\
        a_{2} (y)&= \frac{1}{2}B_{0}+B_{1}+\frac{1}{2}B_{2} + y\frac{1}{2}B_{2} = \frac{1}{12}+\frac{y}{12}.
        \end{align*}
        A short calculation gives
        \[
        \chi_{y}\left(\C P^{2}\right) = a_{0}a_{2}C_{[1,1]}\left(\C P^{2}\right) + \left(a_{1}^{2}-2a_{0}a_{2}\right)C_{[2]}\left(\C P^{2} \right)= y^{2}-y+1.
        \]

        The $\chi_{y}$-genus interpolates several other interesting genera of $\M$. For example, $\chi_{-1}(\M)$ is the Euler characteristic $\chi(\M)$ of $\M$,
        and $\chi_{1}(\M)$ is the \emph{signature} $\sign(\M)$ of $\M$, which in turn is the genus associated to the power series $\frac{x}{\tanh(x)}$.
    \end{exm}

\subsection{The elliptic genus of level \texorpdfstring{$N$}{Lg}}\label{def elliptic genus}

Let $N \geq 2$ be an integer. The \emph{elliptic genus} $\varphi_{N}$ of level $N$ is the genus associated to a certain power series with coefficients in $R = M(\Gamma_{1}(N))$, the ring of modular forms for $\Gamma_{1}(N)$. It is constructed in such a way that the elliptic genus $\varphi_{N}(\M)$ of a manifold of dimension $2n$ is a modular form of weight $n$ for $\Gamma_{1}(N)$.

    We require the \emph{Jacobi theta function}
    \[
    \vartheta(\tau;z) := \sum_{n \in \Z}(-1)^{n}e^{2\pi i \left(n+\frac{1}{2}\right)z}e^{\pi i\left(n+\frac{1}{2}\right)^{2}\tau},
    \]
    with $\tau \in \H$ and $z \in \C$. Using the Poisson summation formula one can show that it satisfies the transformation formulas
    \begin{align}\begin{split}\label{theta transformation}
        \vartheta\left(\frac{a\tau + b}{c\tau + d};\frac{z}{c\tau + d} \right) &= \xi (c\tau + d)^{\frac{1}{2}}e^{\frac{\pi i c z^{2}}{c\tau + d}}\vartheta(\tau;z), \\
        \vartheta'\left(\frac{a\tau + b}{c\tau + d};0\right) &= \xi (c\tau + d)^{\frac{3}{2}}\vartheta'(\tau;0),
        \end{split}
        \end{align}
        for all $\left(\begin{smallmatrix}a & b \\ c & d \end{smallmatrix}\right) \in \SL_{2}(\Z)$, with some eigth root of unity $\xi$ which depends on $a,b,c,d$ but not on $\tau$ and $z$. Here we are abbreviating $\vartheta'(\tau;z) := \frac{d}{dz}\vartheta(\tau;z)$. Furthermore, the Jacobi theta function satisfies the elliptic shifts
        \begin{align}\label{eq theta elliptic transformations}
        \begin{split}
        \vartheta\left(\tau;z+1\right) &= -\vartheta(\tau;z), \qquad
        \vartheta\left(\tau;z+\tau\right) = -e^{\pi i (2z + \tau)}\vartheta(\tau;z).
        \end{split}
        \end{align}
        It can also be written as an infinite product using the \emph{Jacobi triple product identity}
        \begin{align}\label{eq Jacobi triple product}
        \vartheta(\tau;z) = 2i\sin(\pi z)q^{\frac{1}{8}}\prod_{n=1}^{\infty}\left(1-q^{n}\right)\left(1-e^{2\pi i z}q^{n}\right)\left(1-e^{-2\pi i z}q^{n}\right).
        \end{align}
        Note that $\vartheta'(\tau;0) = 2\pi i \eta^{3}(\tau)$ with the \emph{Dedekind $\eta$-function}
        \begin{align}\label{eq eta}
        \eta(\tau) := q^{\frac{1}{24}}\prod_{n=1}^{\infty}(1-q^{n}).
        \end{align}

    We now define the \emph{elliptic genus $\varphi_{N}$ of level $N$} as the genus associated to the normalized power series
    \begin{align}\label{eq Qtau}
    \mathcal Q_{N}(x) := \frac{x}{2\pi i }  \vartheta'(\tau;0)  \frac{\vartheta\left(\tau;\frac{x}{2\pi i}-\frac{1}{N}\right)}{\vartheta\left(\tau;\frac{x}{2\pi i }\right)\vartheta\left(\tau;-\frac{1}{N} \right)} = 1 + a_{1}(\tau)x + a_{2}(\tau)x^{2} + \dots.
    \end{align}
    Note that the elliptic genus $\varphi_{N}$ is only defined for $N \geq 2$. Indeed, since $\vartheta(\tau;z) = 0$ for $z \in \Z\tau + \Z$ the above definition would not make sense for $N = 1$. Keep in mind that $\mathcal Q_{N}(x)$ and $\varphi_{N}$ depend on $\tau \in \H$.

    There are several other useful representations of the power series $\mathcal Q_{N}(x)$. For instance, using the product expansions of $\eta(\tau)$ and $\vartheta(\tau;z)$ stated in \eqref{eq eta} and \eqref{eq Jacobi triple product} above, we can write $\mathcal Q_{N}(x)$ as an infinite product (with $\zeta_{N} = e^{2\pi i/N}$),
    \begin{align*}\label{eq Qtau product}
    \mathcal Q_{N}(x) = x\frac{\left(1-e^{-x}\zeta_{N}\right)}{\left(1-e^{-x}\right)\left(1-\zeta_{N}\right)}\prod_{n=1}^{\infty}\frac{\left(1-e^{-x}\zeta_{N}q^{n}\right)\left(1-e^{x}\zeta_{N}^{-1}q^{n}\right)\left(1-q^{n}\right)^{2}}{\left(1-e^{-x}q^{n}\right)\left(1-e^{x}q^{n}\right)\left(1-\zeta_{N}q^{n}\right)\left(1-\zeta_{N}^{-1}q^{n}\right)},
    \end{align*}
    which is often used as the definition of $\mathcal Q_{N}(x)$ in the literature, see for example Appendix~III of \cite{HBJ}. Many more interesting properties and different representations of $\mathcal Q_{N}(x)$ can be found in Appendix I of \cite{HBJ} and in the theorem in Section~3 of \cite{zagierperiods}, where the power series $\mathcal Q_{N}(x)$ was studied in connection with periods of modular forms.

    We have the following first basic result about the power series $\mathcal Q_{N}(x)$:

    \begin{lemma}
        The coefficients $a_{k}(\tau)$ are modular forms of weight $k$ for $\Gamma_{1}(N)$.
    \end{lemma}

    The lemma is well-known, and can be proved using the transformation behaviour of $\vartheta$ and its derivative stated in \eqref{theta transformation} above. We give a different proof, by showing that the $a_{k}(\tau)$ are Eisenstein series. This explicit representation seems to be less known.

    \begin{lemma}\label{lemma eisenstein}
        We have $a_{k}(\tau) = G_{k,N}(\tau)$, the Eisenstein series defined in equation \eqref{definition eisenstein}.
    \end{lemma}

    \begin{proof}
        By item (vii) of the theorem in Section 3 of \cite{zagierperiods}, we have that
    \begin{align*}
    \mathcal{Q}_{N}(x) = xF_{\tau}\left(\frac{x}{2\pi i},-\frac{1}{N}\right),
    \end{align*}
   with the function
\[
F_{\tau}\left(z_1, z_2\right) := \sum_{n=0}^{\infty}\frac{\xi_2^{-n}}{\xi_1 q^{-n}-1}-\sum_{m=0}^{\infty}\frac{\xi_1^{m}\xi_2}{q^{-m}-\xi_2} \qquad (\xi_j:=e^{2\pi iz_j}).
\]
     On the other hand, by item (iv) of the same theorem from \cite{zagierperiods}, we have the Taylor expansion
    \begin{align*}
    xF_{\tau}\left(\frac{x}{2\pi i},-\frac{1}{N}\right) = 1-N\frac{x}{2\pi i}+\sum_{r,s \geq 0}|r-s|!\left(\frac{1}{2\pi i}\frac{d}{d\tau} \right)^{\min(r,s)}G_{|r-s|+1,1}(\tau)\frac{\left(-\frac{2\pi i}{N} \right)^{s}}{s!}\frac{x^{r+1}}{r!},
    \end{align*}
    where $G_{k,1}(\tau)$ is defined by the Fourier expansion \eqref{eisenstein expansion} if $k > 0$ is even, and $G_{k,1}(\tau)= 0$ if $k$ is odd. Note that we use a different normalization of $G_{k,1}(\tau)$ than \cite{zagierperiods}. We plug in the Fourier expansion \eqref{eisenstein expansion} of $G_{k,1}(\tau)$ and compute its derivative coefficient-wise. We first consider the constant term of $a_{k}(\tau)$. For $k = 1$ it is given by
\begin{align*}
-\frac{N}{2\pi i }+\sum_{\substack{s \geq 0 \\ s \text{ odd}}}\frac{B_{s+1}}{(s+1)!}\left(-\frac{2\pi i}{N}\right)^{s} &= \left(-\frac{N}{2\pi i} \right)\sum_{s \geq 0}\frac{B_{s}}{s!}\left(-\frac{2\pi i }{N} \right)^{s}+\frac{1}{2} \\
&= \left(-\frac{N}{2\pi i} \right)\frac{\left(-\frac{2\pi i}{N} \right)}{e^{-\frac{2\pi i}{N}}-1}+\frac{1}{2} =
\frac{1+\zeta_{N}}{2\left(1-\zeta_{N}\right)},
\end{align*}
where we use that $B_{0} = 1, B_{1} = -\frac{1}{2},$ and $B_{s} =0$ for odd $s > 1$ in the first step, and the definition $\frac{x}{e^{x}-1} = \sum_{s \geq 0}B_{s}\frac{x^{s}}{s!}$ of the Bernoulli numbers in the second step. For $k > 1$ the constant term of $a_{k}(\tau)$ is given by the constant term of $G_{k,1}(\tau)$, which is $\frac{B_{k}}{k!}$. In any case, the constant coefficient of $a_{k}(\tau)$ equals the constant term of the Eisenstein series $G_{k,N}(\tau)$.

Now we compute the $n$-th Fourier coefficient ($n > 0$) of $a_{k}(\tau)$. Noting that $\left(\frac{1}{2\pi i}\frac{\partial}{\partial \tau}\right)^{\ell}q^{n} = n^{\ell}q^{n}$ for $\ell \geq 0$ and $n > 0$,  the $n$-th coefficient of $a_{k}(\tau)$ is given by
\begin{align*}
&-\sum_{s \geq 0}n^{\min(k-1,s)}\sum_{d \mid n}d^{|k-1-s|}\left(1+(-1)^{|k-1-s|+1}\right)\frac{\left(-\frac{2\pi i}{N}\right)^{s}}{s!}\frac{1}{(k-1)!} \\
&= -\frac{1}{(k-1)!}\sum_{d \mid n}\left(\frac{n}{d}\right)^{k-1}\left(\sum_{s \geq 0}\left(1+(-1)^{k-s}\right)\frac{\left(-\frac{2\pi i d}{N}\right)^{s}}{s!} \right) \\
&= -\frac{1}{(k-1)!}\sum_{d \mid n}\left( \frac{n}{d}\right)^{k-1}\left(\zeta_{N}^{-d}+(-1)^{k}\zeta_{N}^{d} \right),
\end{align*}
where we use the definition $e^{x} = \sum_{s \geq 0}\frac{x^{s}}{s!}$ in the last step. We obtain that the $n$-th coefficient of $a_{k}(\tau)$ equals the $n$-th coefficient of $G_{k,N}(\tau)$, which finishes the proof.
    \end{proof}

    Recall that the genus associated to a power series $Q$ can be written as
    \begin{equation*}\label{varphi chern}
    \varphi_{Q}(\M) = \sum_{\lambda \in P(n)}f_{\lambda}C_{\lambda}(\M)
    \end{equation*}
    with $f_{\lambda} \in R_{n}$ and the Chern numbers $C_{\lambda}(\M) \in \Z$. In particular, for the elliptic genus, every $f_{\lambda}$ is a modular form of weight $n$ for $\Gamma_{1}(N)$ and a polynomial in the Eisenstein series $G_{k,N}$. Thus we obtain the following important result.

    \begin{prop} The elliptic genus $\varphi_{N}(\M)$ of an almost complex compact manifold of dimension $2n$ is a modular form of weight $n$ for $\Gamma_{1}(N)$.
    \end{prop}

    \begin{exm}
        To emphasize the explicit nature of the construction of the elliptic genus, we computed the relevant modular forms $f_{\lambda}$ (up to $q^{5}$) for $n = 2$ and $N \in \{2,3\}$:
        \begin{align*}
        \renewcommand{\arraystretch}{1.2}
        \begin{array}{|c|l|l|}
        \hline
        N & \lambda & f_{\lambda}  (\tau)\\
        \hline \hline
        2 & [2] &   - \frac{1}{6} - 4q  - 4q^2- 16q^3- 4q^4-24q^5+ \dots  \\
        & [1,1] &   \frac{1}{12}+ 2q + 2q^2+q^3+ 2q^4+ 12q^5+ \dots    \\
        3 & [2] &
   \frac{1}{4}- 3q- 9q^2- 3q^3- 21q^4-18q^5 + \dots \\
         & [1,1] &  \frac{1}{12}+ q + 3q^2 + q^3+ 7q^4+6q^5+\dots \\
          \hline
        \end{array}
    \end{align*}
    For example, the elliptic genus of level $N = 2$ of an almost complex compact manifold $\M$ of dimension four with Chern numbers $C_{[2]}(\M),C_{[1,1]}(\M) \in \Z$ is given by
    \[
    \varphi_{N}(\M) = f_{[2]}C_{[2]}(\M) + f_{[1,1]}C_{[1,1]}(\M)
    \]
    with
    \[
    f_{[2]} (\tau)=   - \frac{1}{6} - 4q- 4q^2 - 16q^3- 4q^4-24q^5+\dots, \quad f_{[1,1]} (\tau)= \frac{1}{12}+ 2q+ 2q^2+q^3+ 2q^4+12q^5+\dots,
    \]
    which are modular forms of weight two for $\Gamma_{1}(2)$.
    \end{exm}

    \subsection{The values of the elliptic genus at the cusps}\label{value cusps}

    In this subsection we discuss an important result which roughly states that the values at the cusps of the elliptic genus $\varphi_{N}(\M)$ of level $N$ of an almost complex manifold $\M$ are given by certain other interesting genera of $\M$.

    Let $w$ be a primitive $N$-division point of $\C/(\Z\tau + \Z)$, that is, $Nw \in \Z \tau +\Z$ and $N$ is the minimal positive integer with this property. We can write it in the form $w = \frac{k}{N}\tau + \frac{\ell}{N}$ with $k,\ell \in \Z$. Since $w$ is primitive, we may assume that $\gcd(k,\ell) = 1$. To every such $N$-division point $w = \frac{k}{N}\tau + \frac{\ell}{N}$ we associate the cusp $\frac{a}{k}$ of $\Gamma_{1}(N)$, where $a \in \Z$ is chosen (existence follows by Bezout) such that  $a\ell-bk = 1$ for some $b\in\Z$. One can check that the $\Gamma_{1}(N)$-class of the cusp $\frac{a}{k}$ is uniquely determined by this condition, independently of the particular choices of $a,b,k,\ell$ that we make above. However, we remark that a cusp can be represented by several different $N$-division points $w$. For example, the cusp $\infty = \frac{1}{0}$ is represented by the $N$-division point $\frac{1}{N}$, but also by $-\frac{1}{N}$.

    For a primitive $N$-division point $w = \frac{k}{N}\tau +\frac{\ell}{N}$ we consider the normalized power series
   \begin{equation}\label{eq Qtaucd}
   \mathcal Q_{N,(k,\ell)}(x):= e^{-\frac{kx}{N}}\frac{x}{2\pi i}\vartheta'(\tau;0)\frac{\vartheta\left(\tau;\frac{x}{2\pi i}-w \right)}{\vartheta\left(\tau;\frac{x}{2\pi i}\right)\vartheta\left(\tau;-w\right)}
    \end{equation}
  \begin{equation*}
    = x\dfrac{e^{-\frac{kx}{N}}\left(1-e^{-x}\zeta_{N}^{\ell}q^{\frac{k}{N}}\right)}{\left(1-e^{-x}\right)\left(1-\zeta_{N}^{\ell}q^{\frac{k}{N}}\right)}\prod_{n=1}^\infty
\dfrac{\left(1-e^{-x}\zeta_{N}^{\ell}q^{n+\frac{k}{N}}\right)\left(1-e^x\zeta_{N}^{-\ell}q^{n-\frac{k}{N}}\right)\left(1-q^n\right)^2}{\left(1-e^{-x}q^n\right)\left(1-e^xq^n\right)\left(1-\zeta_{N}^{\ell}q^{n+\frac{k}{N}}\right)\left(1-\zeta_{N}^{-\ell}q^{n-\frac{k}{N}}\right)},
   \end{equation*}
    where we are using the product expansions \eqref{eq Jacobi triple product} and \eqref{eq eta}. We let $\varphi_{N,(k,\ell)}$ be the associated genus. Using the transformation rules \eqref{eq theta elliptic transformations} we see that $\mathcal Q_{N,(k,\ell)}(x)$ and $\varphi_{N,(k,\ell)}$ only depend on $k$ and $\ell\pmod{N}$, that is, they only depend on the $N$-division point $w$. Note that for $w = \frac{1}{N}$ we recover the elliptic genus of level $N$, e.g.,
    \[
    \mathcal Q_{N}(x) = \mathcal Q_{N,(0,1)}(x), \qquad \varphi_{N} = \varphi_{N,(0,1)}.
    \]
    It follows from Theorem 6.4 in Appendix I of \cite{HBJ} that for every matrix $A = \left( \begin{smallmatrix}a & b \\ k & \ell \end{smallmatrix}\right) \in \SL_{2}(\Z)$ we have the relation
    \begin{align}\label{eq elliptic genus cusps}
    \varphi_{N}(\M)|_{n}A = \varphi_{N,(0,1)A}(\M) = \varphi_{N,(k,\ell)}(\M)
    \end{align}
    with the weight $n$ slash operator defined in Section~\ref{section modular forms}. Equivalently, this means that the expansion of the elliptic genus $\varphi_{N}(\M)$ at the cusp $\frac{a}{k}$ (represented by some primitive $N$-division point $w = \frac{k}{N}\tau + \frac{\ell}{N}$) is given by the genus $\varphi_{N,(k,\ell)}(\M)$. This fact can be used to prove the following result about the values of the elliptic genus at the cusps.

    \begin{thm}[\cite{HBJ}, Theorem on page 100]\label{special values}
Let $(\M,\J)$ be a compact almost complex manifold with first Chern class $c_1\in H^2(\M;\Z)$ and index $\k0$. Let $N$ be a positive integer dividing $\k0$.
If one represents a cusp $\frac{a}{k}$ of $\Gamma_1(N)$ by a primitive $N$-division point $\frac{k}{N} \tau + \frac{\ell}{N}$, with $0\leq k < N$ and $0\leq  \ell < N$, then the value of $\varphi_N(\M)$ in this cusp equals
$$
\ind\left(L^k\right),\quad \text{if}\quad k>0,
$$
where $L$ is the line bundle with $N\cdot c_1(L)=-c_1 $, and
$$
\frac{\chi_y(M)}{(1+y)^n}, \quad \text{if}\quad k=0 \quad \text{and} \quad y=-e^{\frac{2\pi  i \ell}{N}},\quad\text{with}\quad \gcd(\ell,N)=1.
$$
\end{thm}

  \subsection{The elliptic genus as the index of a virtual vector bundle}\label{vvb}

We first set notation and recall some facts about the Chern
characters of wedge and symmetric products of vector bundles (see
\cite[Subsection 1.5]{HBJ}). Let $E$ be a complex vector bundle of rank
$n$ over a differentiable manifold $\M$. Then $E$ formally splits as
a sum of line bundles $E=L_1\oplus \ldots\oplus L_n$ with
$x_j=c_1(L_j)$. Let $E^*$ be the dual bundle, $\wedge^k E$ (resp.\
$S^k E$) the $k$-th exterior (resp.\ symmetric) power of $E$. 
If we write
$$
\wedge_tE:=\sum_{j=0}^\infty \left(\wedge^j E\right) t^j
\,\,\,\,\,\text{and}\,\,\,\,\wedge_tE^*:=\sum_{j=0}^\infty
\left(\wedge^j E^*\right) t^j\,,
$$
then
\begin{align}\label{exterior}
\operatorname{ch}(\wedge_tE)=\prod_{j=1}^n\left(1+te^{x_j}\right)\,\,&\text{,}\,\,\operatorname{ch}(\wedge_tE^*)=\prod_{j=1}^n\left(1+te^{-x_j}\right)\,\,\,
\text{and}\\ \operatorname{ch}(\wedge_t(E\otimes
\C))&=\prod_{j=1}^n\left((1+te^{x_j})(1+te^{-x_j})\right)\,.\nonumber
\end{align}
Analogously, if we set
$$
S_tE:=\sum_{j=0}^\infty \left(S^j E\right) t^j
\,\,\,\,\,\text{and}\,\,\,\,S_tE^*:=\sum_{j=0}^\infty \left(S^j
E^*\right) t^j\,,
$$
then
\begin{align}\label{symmetric power}
\operatorname{ch}(S_tE)=\prod_{j=1}^n\dfrac{1}{1-te^{x_j}}
\,\,&\text{,}\,\,\operatorname{ch}(S_tE^*)=\prod_{j=1}^n\dfrac{1}{1-te^{-x_j}}\,\,\,
\text{and} \\ \operatorname{ch}(S_t(E\otimes
\C))&=\prod_{j=1}^n\dfrac{1}{(1-te^{x_j})(1-te^{-x_j})}\,.\nonumber
\end{align}

Let $(\M^{2n}, \J)$ be an almost complex manifold, with formal
decomposition of its tangent bundle $T:=T\M=E_1\oplus\ldots \oplus
E_n$ as sum of line bundles and formal variables $x_j:=c_1(E_j).$
Let $N$ be a positive integer that divides the index $\k0$ of $(\M,
\J)$ and let $k, \ell$ be two integers such that $\frac{k}{N}\tau +
\frac{\ell}{N}$ is a  primitive $N$-division point of $\C/(\Z\tau +
\Z)$ with $\tau \in \H\,.$ Recall that the elliptic genus
$\varphi_{N, (k,\ell)}$ is defined to be the genus belonging to the
power series of $\mathcal Q_{N, (k, \ell)}$, see equation \eqref{eq Qtaucd}. Formally
\begin{align}\label{eq normalized elliptic genus}
\begin{split}
\varphi_{N, (k,\ell)}(\M) &=\int_\M \mathcal Q_{N, (k ,\ell)}(x_1)\cdot\ldots\cdot \mathcal Q_{N, (k, \ell)}(x_n) \\
&=m(q)^n\int_\M \widetilde{\mathcal Q}_{N, (k,
\ell)}(x_1)\cdot\ldots\cdot \widetilde{\mathcal Q}_{N, (k,
\ell)}(x_n)=m(q)^n \widetilde{\varphi}_{N,(k,\ell)}(\M),
\end{split}
\end{align}
where
\begin{align*}
m(q):=\dfrac{1}{\left(1-\zeta_{N}^{\ell}
q^{\frac{k}{N}}\right)}\prod_{r=1}^\infty
\dfrac{\left(1-q^{r}\right)^2}{\left(1-\zeta_{N}^{\ell}
q^{r+\frac{k}{N}}\right)\left(1-\zeta_{N}^{-\ell}q^{r-\frac{k}{N}}\right)}
\end{align*}
and $\widetilde{\mathcal Q}_{N, (k, \ell)}(x):=\frac{\mathcal Q_{N,
(k, \ell)}(x)}{m(q)}$. The function
$\widetilde{\varphi}_{N,(k,\ell)}(\M)$ is called the
\emph{normalized elliptic genus of $\M$ of level $N$}. It is
holomorphic on $\H$ and transforms like a modular form of weight
zero, but it has terms of negative index in the Fourier expansions
at some cusps. We write
$\widetilde{\varphi}_{N,(k,\ell)}(\M)(q^{N})$ for the normalized
elliptic genus with $q$ replaced by $q^{N}$. It has the advantage
that in its Fourier expansion only integral powers of $q$ appear.

The Atiyah-Singer Theorem, together with \eqref{exterior}, \eqref{symmetric power}, and \eqref{eq Qtaucd}, imply that \\
$\widetilde{\varphi}_{N,(k,\ell)}(\M)$ can be regarded as the
topological index of an infinite tensor product, namely
\begin{equation}\label{infinite tensor}
\widetilde{\varphi}_{N, (k, \ell)}(\M)=\ind\left(L^k\otimes R_{N, (k,
\ell)}(q)\right)
\end{equation}
where $R_{N, (k, l)}(q)$ is the virtual bundle
\begin{equation}\label{virtual bundle}
R_{N, (k
,\ell)}(q)=\left(\wedge_{-\zeta_{N}^{\ell}q^{\frac{k}{N}}}T^*\otimes
\bigotimes_{r=1}^\infty
\wedge_{-\zeta_{N}^{\ell}q^{r+\frac{k}{N}}}T^*\otimes
\wedge_{-\zeta_{N}^{-\ell}q^{r-\frac{k}{N}}}T\otimes S_{q^{r}}T^*
\otimes S_{q^{r}}T\right)
\end{equation}
with Chern character equals to
\begin{align*}
\operatorname{ch}\left(R_{N, (k, \ell)}(q)\right)=
  \prod_{j=1}^n \left(\left(1-e^{-x_j}\zeta_{N}^{\ell}q^{\frac{k}{N}}\right)\prod_{r=1}^\infty
\dfrac{\left(1-e^{-x_j}\zeta_{N}^{\ell}q^{r+\frac{k}{N}}\right)\left(1-e^{x_j}\zeta_{N}^{-\ell}q^{r-\frac{k}{N}}\right)}{\left(1-e^{-x_j}q^r\right)\left(1-e^{x_j}q^r\right)}\right)
 \end{align*}
and $L$ is the line bundle with Chern class equals to
$c_1(L)=-\frac{c_1}{N}$.
We remark that in the statement corresponding to \eqref{infinite
tensor} in \cite{HBJ}, Appendix III, equation (16) on page 176, there
are some minor typos which we fixed here.

\begin{rmk} Conceptually, the elliptic genera of level $N$ can be
interpreted as the equivariant Hirzebruch $\chi_{y}$-genus
(evaluated at a $N$-root of unity) of the free loop space
$\mathcal{L}\M\,.$ On the loop space $\mathcal{L}\M,$ there is a
canonical $S^1$-action having $\M$, the space of constant loops, as
its fixed point set. The tangent space $T_P(\mathcal{L}\M)$ at a
constant loop $P\in \M$ is the loop space $\mathcal{L}(T_P \M).$ Any
loop in $\mathcal{L}(T_P \M)$ admits a Fourier expansion with
coefficients in $T_P\M$ and as a consequence we obtain the following
weight decomposition
$$
T_P(\mathcal{L}\M)\cong \mathcal{L}(T_P \M)=T_P\M\oplus \sum_{r\in
\N}q^r (T_P \M\otimes \C),
$$
where $S^1$ acts on the $r$-th summand with weight $r\in \N\,.$ Here
$q$ denotes a formal variable that keeps track of the circle action
on each summand. If we apply the equivariant Atiyah-Singer index
theorem formally for the canonical circle action on $\mathcal{L}\M$
to compute the equivariant genus associated to
$$
\dfrac{x}{1-e^{-x}}(1+ye^{-x})\,\,\,\,\,\,\text{with
}\,\,\,\,\,\,y=-\zeta_N,
$$
we obtain the heuristic formula
\begin{multline*}
\int_{\M} \prod_{j=1}^n x_j\cdot \left(\prod_{r=-\infty}^{\infty}
\dfrac{1-e^{-x_j}\zeta_{N}q^{r}}{1-e^{-x_j}q^{r}} \right)
\\=\int_{\M} \prod_{j=1}^n x_j\cdot
\left(\dfrac{1-e^{-x_j}\zeta_{N}}{1-e^{-x_j}}\prod_{r=1}^\infty
\dfrac{(1-e^{-x_j}\zeta_{N}q^{r})(1-e^{x_j}\zeta_{N}^{-1}q^{r})\zeta_N}{(1-e^{-x_j}q^r)(1-e^{x_j}q^r)}\right)
\end{multline*}
which (up to the formal factor
``$\zeta_N^{\,\sum_{r=1}^\infty\,1}$'') coincides with the
normalized elliptic genus $\widetilde{\varphi}_{N, (0, 1)}(\M)(q)$
when $q=e^{2\pi i \tau}$ with $\tau \in \H\,.$

If $N=2,$ then we obtain the equivariant signature of the loop space
as initially described by Witten for spin manifolds in the context
of quantum field theory \cite{W1, W2}. We refer the reader to
\cite[Section 7.4]{HBJ} for a more detailed treatment of elliptic
genera as the $\chi_y$ genus of the loop space of a manifold, where
the convergence of the heuristic formula is more carefully treated.
\end{rmk}

\subsection{The type of the action}\label{type of action}

In this subsection we recall what the type of the action of $\ac$ is and prove how it is related to the index of $(\M,\J)$ (see \cite[Appendix III, Sections\ 6 and 8]{HBJ}).

\begin{defin}\label{balanced}\cite[page\ 179]{HBJ}
Given an integer $N$, the $S^1$-action on $\ac$ is said to be \textit{$N$-balanced} if for every $P\in \fix$ the residue class of $w_{1}(P)+\cdots+w_{n}(P)\pmod{N}$ does not depend on $P$.
Given an $N$-balanced action, the common residue class of $w_{1}(P)+\cdots+w_{n}(P)$ modulo $N$ is called the \textit{type of the action}.
\end{defin}
In order to understand the geometric meaning of the previous definitions we need to recall notions about the first Chern class of $(\M,\J)$, its equivariant extension and the index of $(\M,\J)$.
Let $c_1\in H^2(\M;\Z)$ be the first Chern class of the tangent bundle of $(\M,\J)$ and $c_1^{S^1}\in H^2_{S^1}(\M;\Z)$ the equivariant first Chern class of $\ac$. We recall that, given a fixed point $P$ and the restriction map $r_P\colon H^2_{S^1}(\M;\Z)\to H^2_{S^1}(\{P\};\Z)$, there exists $x \in H^2_{S^1}(\{P\};\Z)$ such that $H^2_{S^1}(\{P\};\Z)\simeq \Z[x]$ and
\begin{equation}\label{efcc}
r_P\left(c_1^{S^1}\right)=(w_{1}(P)+\cdots+w_{n}(P))x\,.
\end{equation}
 Henceforth we simply denote the restriction of an equivariant cohomology class $c\in H^*_{S^1}(\M;\Z)$ to a fixed
point $P$ by $c(P)$.

Next we recall what the index of $(\M,\J)$ is.
\begin{defin}\label{def index}\cite{S}
Given an almost complex manifold $(\M,\J)$, the \textit{index} of $(\M,\J)$ is the largest integer $\k0$ such that, modulo torsion elements, $c_1=\k0 \eta$
for some nonzero $\eta \in H^2(\M;\Z)$.
\end{defin}
Therefore $\k0$ is zero exactly if $c_1$ is torsion, and is otherwise the largest integer such that $c_1/\k0 \in H^2(\M;\Z)$.
\begin{rmk}\label{properties index}
\begin{itemize}[leftmargin=20pt]
\item[\vphantom{workaround}]
\item[(1)] In \cite{S} the third author proved that, given a compact almost complex manifold of dimension $2n$ endowed with a circle action that preserves $\J$ and with isolated fixed points, if the Todd genus of $(\M,\J)$ is nonzero, then the index is bounded above by $n+1$.
\item[(2)] If $\M$ is simply connected, then Hurewicz theorem asserts that $\pi_2(\M)\simeq H_2(\M)$. By the universal coefficient theorem, this implies that the group of homomorphisms $\homo(\pi_2(\M);\Z)\simeq \homo(H_2(\M);\Z)$ is isomorphic to $H^2(\M;\Z)$. It follows that for simply connected almost complex manifolds
the index coincides with the non-negative integer $D$ such that $\langle c_1 , \pi_2(\M)\rangle = D \Z$. If $D\neq 0$, then this positive integer is also known as the \textit{minimal Chern number} of $\M$ (see \cite[Definition 6.4.2]{DS}; observe that the existence of a symplectic structure in \cite[Definition 6.4.2]{DS} is not needed).
\item[(3)] If $(\M,\omega)$ is a compact symplectic manifold endowed with a Hamiltonian circle action with isolated fixed points (see Definition \ref{hamiltonian t action}), then it is simply connected and its Todd genus does not vanish. By part (2) of this remark it follows that for such manifolds the index coincides with the minimal Chern number (see also \cite[Remark 3.4]{S})
\end{itemize}
\end{rmk}
The next proposition, which is already known (see \cite[Appendix III, Section \ 8]{HBJ} and \cite[Section \ 9]{BT}), gives a relation between the index and the action. For the sake of completeness we include a proof here.
\begin{prop}\label{balanced index}
Let $(\M,\J)$ be a compact almost complex manifold of dimension $2n$ endowed with an $S^1$-action preserving $\J$ and let $\k0$ be its index and $N$ a positive integer dividing $\k0$.
 Then the $S^1$-action is $N$-balanced.
 \end{prop}
\begin{rmk}
Observe that if $\k0=0$, then the proposition says that the action is balanced for every $N$, hence the sum of the weights at each fixed point $P$ is independent
of $P$. Moreover, if the fixed point set is discrete, then Lemma 2.13
in \cite{GPS} gives that
 this sum is always zero.
\end{rmk}
\begin{proof}[Proof of Proposition \ref{balanced index}]
By \eqref{efcc} it is sufficient to prove that for all pairs of distinct fixed points $P_1,P_2\in \fix$, one has
$$
\frac{c_1^{S^1}(P_1)-c_1^{S^1}(P_2)}{N\cdot x}\;\; \in \;\;\Z.
$$
First suppose that $P_1$ and $P_2$ belong to the same connected component $\M_0$ of $\fix$. Observe that the weights of the
$S^1$-action at a fixed point $s\in \M_0$ on the component of $T\M$ tangent to $\M_0$ are zero, whereas those in the normal component to $\M_0$ in $\M$ do not depend
on $s\in \M_0$. By \eqref{efcc} it follows that $c_1^{S^1}(P_1)=c_1^{S^1}(P_2)$.

Now suppose that $P_1$ and $P_2$ are fixed points belonging to different connected components of $\fix$.
Let $\gamma\colon [0,1]\to \M$ be a continuous
path such that $\gamma(0)=P_1$, $\gamma(1)=P_2$ and the image of the interval $(0,1)$ avoids other fixed points;
this is possible since the normal bundle to a fixed component is always of positive even dimension.
Rotating this path using the $S^1$-action one obtains a topological two-sphere $S$ with fixed points given by $P_1$ and $P_2$.
Assume that the common stabilizer of the points of $S$ is $\Z_k$ (observe that $k$ could be one, in which case the common stabilizer if trivial).
Let $\iota \colon S^2\to \M$ be an equivariant map from a smooth sphere $S^2$, where the action on $S^2$ is given in cylindrical coordinates by
$S^1 \ni \lambda=e^{\iota \theta'} * (\theta, h)=(\theta + k \theta',h)$, and such that the image is exactly $S$.
Let $Q_1$ and $Q_2$ be the two fixed points on $S^2$ with $\iota(Q_j)=P_j$, $j\in \{1,2\}$, and assume that  the weight of the $S^1$-action on $TS^2|_{Q_1}$ is $k$
and that on $TS^2|_{Q_2}$ is $-k$.
If we pull-back the tangent bundle of $\M$ to $S^2$, then the weights of the $S^1$-action on $\iota^*(T\M)$ at the fixed point $Q_j$ coincide with those at $P_j\in \M$, for $j\in \{1,2\}$.
Thus \eqref{efcc} implies that $\iota^*(c_1^{S^1})(Q_j)=c_1^{S^1}(P_j)$ and we have
\begin{align}\label{eq_int}
\Z\ni \int_S c_1&=\int_{S^2}\iota^*(c_1)=\int_{S^2}\iota^*\left(c_1^{S^1}\right)=\frac{\iota^*\left(c_1^{S^1}\right)(Q_1)-\iota^*\left(c_1^{S^1}\right)(Q_2)}{k\, x}
\nonumber \\&=\frac{c_1^{S^1}(P_1)-c_1^{S^1}(P_2)}{k\,x},
\end{align}
where the second equality follows by degree reasons and the last equality from the localization formula \eqref{ABBV}.

If $\k0$ is zero, or equivalently if $c_1$ is torsion, then the integral on the left-hand side is zero. Hence $c_1^{S^1}(P_1)=c_1^{S^1}(P_2)$ for every pair of fixed points $P_1$ and $P_2$, and
the action is balanced for every $N$.

Otherwise, by the definition of $\k0$ and the assumption that $N$ divides $\k0$, we have that $c_1=N \eta$ for some non-zero $\eta\in H^2(\M;\Z)$. From
\eqref{eq_int} we have that
$$
\Z \ni \int_S \eta = \int_S \frac{c_1}{N}=\frac{c_1^{S^1}(P_1)-c_1^{S^1}(P_2)}{N\cdot k\,x}\,,
$$
and the claim follows.
\end{proof}

\begin{rmk}
\begin{itemize}[leftmargin=20pt]
\item[\phantom{workarounf}]
\item[(1)]
In \cite[Lemma 2.3]{GPS} it was proved that if $c_1$ is torsion or, equivalently, if $\k0=0$, and in addition the fixed points of the action are isolated, then
the sum of the weights at each fixed point is indeed zero, and so is the type of the action.
\item[(2)] Let $\mathcal{M}_k$ be the closure of the set of points with stabilizer $\Z_k$. Then from the proof of the proposition it follows that, for all the fixed points in $\mathcal{M}_k$, the $S^1$-action is $N\cdot k$ balanced on $\mathcal{M}_k$.
\end{itemize}
\end{rmk}

\subsection{The rigidity of the elliptic genus}\label{section rigidity}
Let $(\M, \J, S^1)$ be an almost complex manifold endowed with a
circle action compatible with the almost complex structure $\J.$ Let
$N$ be a positive integer that divides the index $\k0$ of $(M, \J)$.
In this subsection, we define what it means for the elliptic genus
$\varphi_N(\M)$ of level $N$ to be rigid, but first we recall what
rigidity means for an equivariant bundle.

\begin{defin}
An $S^1$-equivariant bundle $V$ is called \emph{rigid} if its
equivariant index $\ind_{S^1}(V)$ is independent of $t$, hence it
lies in $\mathbb{Z}\subset \Z[t,t^{-1}]$.
\end{defin}
Note that, since taking $t=1$ corresponds to taking the trivial
$S^1$-representation, rigidity of an equivariant bundle $V$ implies
that
$$
\ind_{S^1}(V)=\ind_{S^1}(V)|_{t=1}=\ind(V).
$$
Now let $R_{N, (0,1)}(q)$ be the virtual bundle already defined in
equation~\eqref{virtual bundle} for $(k, l)=(0, 1)$ by
\begin{equation*}
R_{N, (0,1)}(q):=\left(\wedge_{-\zeta_{N}}T^*\otimes
\bigotimes_{r=1}^\infty \wedge_{-\zeta_{N}q^{r}}T^*\otimes
\wedge_{-\zeta_{N}^{-1}q^{r}}T\otimes S_{q^{r}}T^* \otimes
S_{q^{r}}T\right)\,.
\end{equation*}
We can write
$$
R_{N, (0,1)}(q)=\sum_{j\geq 0}R_j\, q^j
$$
for some finite dimensional virtual bundles $R_j$ defined for every
non-negative integer $j\,.$ For instance
\begin{align*}
R_0=\wedge_{-\zeta_{N}}T^*\,,\,\, \text{and}\quad
R_1=\wedge_{-\zeta_{N}}T^*\otimes\left(\left(1-\zeta_N\right)T^*+\left(1-\zeta_N^{-1}\right)T\right)\,.
\end{align*}
Recall that the normalized elliptic genus $\tilde{\varphi}_{N, (0,
1)}$ of level $N$ is defined as the index of the virtual bundle
$R_{N, (0,1)}(q)$ and the elliptic genus $\varphi_{N}(\M)$ of level
$N$ is defined by
$$
\varphi_{N}(\M):=\varphi_{N, (0,
1)}(q)=m(q)^{\frac{\dim{(\M)}}{2}}\tilde{\varphi}_{N, (0,
1)}=m(q)^{\frac{\dim{(\M)}}{2}}\sum_{j\geq 0}\ind(R_j)\,q^j,
$$
where
$$
m(q):=\dfrac{1}{1-\zeta_N}\prod_{r=1}^\infty
\dfrac{\left(1-q^r\right)^2}{\left(1-\zeta_Nq^r\right)\left(1-\zeta_N^{-1}q^r\right)}\,.
$$
We define the \textit{equivariant elliptic genus
$\varphi_N(\M, t)$ of level $N$} as

\begin{equation*}
\varphi_{N}(\M, t):=m(q)^{\frac{\dim{(\M)}}{2}}\sum_{j\geq
0}\ind_{S^1}(R_j)\,q^j\,.
\end{equation*}

\begin{defin}
The equivariant elliptic genus $\varphi_N(\M, t)$ of level $N$  is
\textit{rigid} if the equivariant bundles $R_j$ are rigid for every
$j\,.$
\end{defin}
Note that the rigidity of the equivariant elliptic genus
$\varphi_N(\M, t)$ just means that
$$
\varphi_N(\M,t)=\varphi_N(\M, t)|_{t=1}=\varphi_N(\M).
$$

If $N=2$ and $(\M, \J, S^1)$ is a spin manifold, then the rigidity of
the equivariant elliptic genus $\varphi_2(\M, t)$ was conjectured by
Witten \cite{W1, W2} following theoretical physics ideas on
the loop space of $\M$. The rigidity of $\varphi_2(\M, t)$ was first
proven by Taubes \cite{Ta} who made Witten's program rigorous  and
later by Bott and Taubes \cite{BT} who simplified Taubes original
approach using the language of equivariant index theory and
equivariant cohomology. The following theorem due to Hirzebruch
generalizes the Rigidity Theorem by Bott, Taubes and Witten to
almost complex manifolds that are not necessarily spin.

\begin{thm}[\cite{HBJ}, theorem on page 181]\label{rigidity}
$\;$\\
Let $(\M,\J)$ be a compact almost complex manifold with index $\k0$.
Suppose that $(\M,\J)$ is endowed with an $S^1$-action preserving
$\J$. Then for every positive integer $N$ dividing $\k0$ the
elliptic genus $\varphi_N(\M,t)$ of level $N$ is rigid, hence it
equals the (non-equivariant) elliptic genus $\varphi_N(\M)$.
If the type of the action is not zero  $\pmod{N}$, then
$\varphi_N(\M)\equiv 0$.
\end{thm}

\section{From number theory to geometry}\label{number theory to geometry}

\subsection{From the index to the Betti numbers: First consequences}\label{number theory to geometry I}

The next result is an easy consequence of Theorem \ref{rigidity} and identifies a category of manifolds for which certain elliptic genera always vanish.
\begin{prop}\label{toric vanishing}
Let $(\M,\omega,\psi)$ be a symplectic toric manifold with index $\k0$. Then for every integer $N\geq 2$ dividing $\k0$, the elliptic genus $\varphi_N(\M)$ of level $N$ vanishes identically.
\end{prop}

\begin{proof}
Let $N$ be an integer dividing the index $\k0$. By Theorem \ref{rigidity} it is enough to find a subcircle $S^1$ of the torus $\T$ acting on $\M$ such that the type of the $S^1$-action is not zero $\pmod{N}$. Let $P$ be a fixed point of the $\T$-action on $\M$, the latter being of dimension $2n$. Modulo a GL$_n(\Z)$-transformation we can assume that the weights of the $\T$ action at $P$ are the vectors in the standard basis $x_1,\ldots,x_n$ of $\ts\simeq (\R^n)^*$, namely $x_j(\xi_1,\ldots,\xi_n)=\xi_j$ for every $j\in\{1,\ldots,n\}$ and every $(\xi_1,\ldots,\xi_n)\in \ft\simeq \R^n$. Any circle subgroup $S^1$ of $\T$ is determined by a vector $\alpha=(\alpha_1,\ldots,\alpha_n)$, where $\alpha_j\in \Q$ for every $j$, and
$S^1=\exp\{t\cdot (\alpha_1,\ldots,\alpha_n): t\in \R\}\subset \T$. By rescaling the rational coordinates $\alpha_j$ we can assume that they are indeed integral and that the vector $\alpha$ is primitive
in the lattice $\ell^*\simeq (\Z^n)^*$, namely that if $\alpha=m\cdot \beta$ for some $m\in \Z$ and $\beta\in \ell^*$, then $m=\pm 1$.
If we restrict the $\T$ action to that of the above circle, the weights at $P$ of the induced $S^1$-action are given by
$x_j(\alpha)=\alpha_j$ for every $j\in \{1,\ldots,n\}$. Therefore, in order to prove the proposition, it is enough to find a primitive $\alpha\in \ell^*$ such that $\alpha_1+\cdots + \alpha_n$
is not zero $\pmod{N}$.
This is easily seen to be always possible, as if the picked $\alpha=(\alpha_1,\ldots,\alpha_n)$ satisfying $\alpha_1+\cdots +\alpha_n\equiv 0 \pmod N$, it would be enough to replace
it with $(\alpha_1+1,\ldots,\alpha_n)$, and if the latter were not primitive, dividing the coordinates by the common divisor would yield the desired primitive vector.
\end{proof}

The next result is an application of Theorem \ref{special values}.
\begin{prop}\label{divisibility chiy genus}
Let $(\M,\J)$ be a compact almost complex manifold with index $\k0$. Suppose that for every integer $N\geq 2$ dividing $\k0$, the elliptic genus $\varphi_N(\M)$ of level $N$ vanishes identically. Then $$\sum_{j=0}^{\k0-1}(-y)^j \quad \text{divides} \quad \chi_y(\M).$$
\end{prop}
\begin{proof}
If $\k0=1$ there is nothing to prove.
Thus suppose that $\k0\geq 2$.
As $$(-1)^{\k0-1}\sum_{j=0}^{\k0-1}(-y)^j  = \prod_{\ell=1}^{\k0-1}\left(y+e^{\frac{2\pi \il \ell}{\k0}}\right),$$ it is enough to prove that $\chi_y(\M)=0$ when $y=-e^{\frac{2\pi \il l}{\k0}}$, for all $\ell\in\{1,\ldots,\k0-1\}$.

If $\gcd(\ell,\k0)=1$ then Theorem \ref{special values}, together with the assumption that $\varphi_{\k0}(\M)=0$, imply that $\chi_y(\M)=0$ for $y=-e^{\frac{2\pi \il \ell}{\k0}}$.

If $\gcd(\ell,\k0)=a>1$ then it is sufficient to observe that, for $N:=\frac{\k0}{a}$ and $\ell':=\frac{\ell}{a}$, one has $e^{\frac{2\pi \il \ell}{\k0}}=e^{\frac{2\pi \il \ell'}{N}}$. Since by assumption
$\varphi_N(\M)$ vanishes, Theorem \ref{special values}  gives that $\chi_y(\M)=0$ for $y=-e^{\frac{2\pi \il \ell'}{N}}=-e^{\frac{2\pi \il \ell}{\k0}}$.
\end{proof}
In Subsection \ref{closer look}, we specialize to the case in which the manifold is symplectic and is endowed with a Hamiltonian circle action with isolated fixed points  to conclude that if index is maximal, namely $\k0=n+1$, it is sufficient to assume that the elliptic genus of level $n+1$ vanishes to conclude that $\chi(\M)=\sum_{j=0}^n (-y)^j$. We indeed prove that $\M$ is also
homotopy equivalent to $\C P^n$.

We are now ready to prove Corollary \ref{Betti toric}.
\begin{proof}[Proof of Corollary \ref{Betti toric}]\label{proof corollary}
Combining Propositions \ref{toric vanishing} and \ref{divisibility chiy genus} we obtain that for a symplectic toric manifold of index $\k0$, the polynomial $\sum_{j=0}^{\k0-1}(-y)^j$
divides $\chi_y(\M)$. For the claim \eqref{corollary toric} of Corollary \ref{Betti toric} what is left to prove is that, for a symplectic toric manifold of dimension $2n$, one has
\begin{equation}\label{chiy toric}
\chi_y(\M)=\sum_{j=0}^{n}b_{2j}(\M)(-y)^j\,,
\end{equation}
and then replacing $-y$ with $y$ yields the desired claim.
This fact follows from two observations: the first is that, picking a circle subgroup $S^1$ of the torus $\T$ acting on $\M$ such that the fixed points of the $S^1$-action are
the same of those of the $\T$ action,  gives $\chi_y(\M)=\sum_{j=0}^{n}N_{j}(-y)^j$, where $N_j$ is the number of fixed points with $j$ negative weights of the $S^1$-action (see for instance \cite[Section \ 5.7]{HBJ} and \cite[Section 3]{GS}); the second is that the number $N_j$ is precisely $b_{2j}(\M)$.

        If $\k0 = n+1$, then, as we are assuming that $\M$ is connected, $b_0(\M)=1$ and \eqref{corollary toric} implies that $\mathbf{b}=(1,1,\ldots,1)$. Since, as observed above,
        $N_j=b_{2j}(\M)$ for every $j$, we have that the total number of fixed points is $\sum_{j}N_j=\sum_{j}b_{2j}(\M)=n+1$. On a symplectic toric manifold the number of fixed
        points corresponds exactly to the vertices of the Delzant polytope $\psi(\M)$, which is of dimension $n$. However the only Delzant polytope of dimension $n$ with $n+1$ vertices
        is --up to $\GL_m(\Z)$-transformations-- the smooth simplex of dimension $n$. By Delzant's Theorem \cite{D} we obtain the desired claim.

        If $\k0 = n$, then \eqref{corollary toric} tells us that the polynomial $\sum_{j=0}^{n-1}(-y)^{j}$ divides $\chi_{y}(\M)$, i.e.
        \[
        \chi_{y}(\M) = (\ell(-y) + m)\sum_{j=0}^{n-1}(-y)^{j}.
        \]
        for some $\ell,m \in \Z$. Since $b_0(\M)=b_{2n}(\M)=1$, we obtain that $\ell=m=1$ and hence
        \[
        \chi_{y}(\M) = (-y + 1) \sum_{j=0}^{n-1}(-y)^{j} = 1+ 2(-y)+\cdots + 2(-y)^{n-1}+(-y)^{n}.
        \]
        From \eqref{chiy toric} it follows that $\mathbf{b}=(1,2,\ldots,2,1)$.

        If $\k0 = n-1$, then we obtain as before that
        \[
        \chi_{y}(\M) = (\ell(-y)^{2} + m(-y) + r) \sum_{j=0}^{n-2}(-y)^{j}
        \]
        for some $\ell,m,r\in \Z$, and from \eqref{chiy toric} and the fact that $b_{0}(\M) = b_{2n}(\M) = 1$ we infer $\ell = r= 1$. By multiplying out the above expression, we obtain the stated vector $\mathbf{b}$ of Betti numbers.

        If $\k0 = n-2$, we proceed as before and obtain $$\chi_{y}(\M) = (\ell(-y)^{3}+m(-y)^{2}+r(-y)+s) \sum_{j=0}^{n-3}(-y)^{j},$$
                for some $\ell, m, r, s\in\Z$.
                Since $b_0(\M)=b_{2n}(\M)=1$, \eqref{chiy toric} implies that
        $\ell=s=1$. Hence we get
        \begin{align*}
        \chi_y(\M) &= (-y)^{n} + (1+m)(-y)^{n-1} + (1+m+r)(-y)^{n-2} + (2+m+r)(-y)^{n-3}  + \dots \\
                        & \quad \dots + (2+m+r)(-y)^{3} + (1 + m+ r)(-y)^{2} + (1+r)(-y) + 1.
        \end{align*}
        The symmetry $b_{2}(\M) = b_{2n-2}(\M)$  further yields $m = r$ and the desired vector $\mathbf{b}$ of Betti numbers. Finally, replacing the variable $y$ with $-y$ finishes the proof.
    \end{proof}

\subsubsection{An application to smooth reflexive polytopes}\label{section reflexive}

We would like to give a straightforward application of Corollary \ref{Betti toric} to smooth (or Delzant) reflexive polytopes.
In Definition \ref{delzant} it is recalled what smooth (or Delzant) polytopes are. A particularly interesting subclass of such combinatorial objects is given by smooth reflexive polytopes. \
\begin{defin}\label{reflexive}
An $n$-dimensional polytope $\Delta$ is called {\it reflexive} if all its vertices belong to $\Z^n\subset \R^n$, it contains zero in its interior, and such that
\begin{equation*}\label{refl eq}
\Delta=\bigcap_{j=1}^k\{x\in \R^n : \langle x , \nu_j \rangle \leq 1\}\,,
\end{equation*}
where the $\nu_j\in \Z^n$ are the primitive outward normal vectors to the hyperplanes defining the facets, for $j\in\{1,\ldots,k\}$.
A {\it smooth reflexive} polytope is a polytope that is at the same time smooth (or Delzant) and reflexive.
\end{defin}

The importance of the smoothness condition relies in the fact that, via the Delzant Theorem, many combinatorial features of $\Delta$ have a topological counterpart in the corresponding manifold.
Hence it is often the case that one can use the topology of the manifold to derive combinatorial properties of the associated polytope.
For instance let $\mathbf{f}$ be the $f$-vector of $\Delta$, namely $\mathbf{f}=(f_0,\ldots,f_n)$, where $f_j$ is the number of faces of $\Delta$ of dimension $j$.
Then the $h$-vector of $\Delta$ is defined to be $\mathbf{h}=(h_0,\ldots,h_n)$, where
$$
h_j=\sum_{r=0}^j (-1)^{j-r}\binom{n-r}{n-j}f_{n-r}\quad\mbox{for all}\quad j\in\{0,\ldots,n\}.
$$
The following result is a first instance of the combinatorial--topological correspondence mentioned above; for a proof see for instance \cite[Lemma 3.8]{GHS} and the references therein.
.
\begin{lemma}\label{h=b}
Let $(\M,\omega,\psi)$ be a symplectic toric manifold of dimension $2n$ and $\Delta=\psi(\M)$. Then the $h$-vector of $\Delta$ is exactly the vector of even Betti numbers of $\M$, namely
$h_j=b_{2j}(\M)$ for every $j\in\{0,\ldots,n\}$.
\end{lemma}

Given a smooth polytope $\Delta$ and the corresponding symplectic toric manifold $(\M,\omega,\psi)$, with $\psi(\M)=\Delta$, the reflexivity condition on $\Delta$
corresponds exactly to the so-called monotonicity\footnote{The monotonicity condition is sometimes also referred to as the \emph{symplectic Fano condition}, and more
generally it means $c_1=\lambda[\omega]$ for some $\lambda \in \R$. However one can prove (see for instance \cite[Lemma 5.2]{GHS}) that if $(\M,\omega)$ admits
a Hamiltonian action, then $\lambda$ must be positive. Hence for symplectic toric manifolds satisfying $c_1=\lambda[\omega]$ one can rescale the symplectic form
to obtain $c_1=[\omega]$.} of $\M$, namely $c_1=[\omega]$ (see for instance \cite[Proposition 3.10]{GHS}). This allows one to translate even more properties of $\M$
into those of $\Delta$, as the following lemma illustrates. Before stating it, we recall that given a rational segment $e=(v_1,v_2)$ between two points $v_1,v_2$ in $\R^n$,
namely
\begin{equation}\label{rational segment}
v_2-v_1=l(e)w\quad\text{for some}\;\; l(e)\in \R^+\;\;\text{and}\;\;w\in \Z^n\,,
\end{equation}
its {\it affine length} is the positive number $l(e)$ defined by the above displayed equation if $w$ is the unique primitive vector in $\Z^n$ satisfying \eqref{rational segment}.
We recall that given a symplectic toric manifold $(\M,\omega,\psi)$, the edges of $\psi(\M)$ correspond, via the moment map, to symplectic spheres in $\M$. As the set of these spheres
generate $H_2(\M;\Z)$, the monotonicity condition $c_1=[\omega]$ has this important consequence (see \cite[Proposition 5.4]{GHS})
\begin{lemma}
Let $\Delta$ be a smooth reflexive polytope and $(\M,\omega,\psi)$ the corresponding monotone symplectic toric manifold of index $\k0$. Then we have
$$
\k0 = \gcd\{l(e): e\in E\},
$$
where $E$ is the set of edges of $\Delta$.
\end{lemma}
Roughly speaking, for a monotone symplectic toric manifold the index becomes ``visible''  in the associated polytope. These two of known lemmas, together with the Delzant correspondence and Corollary \ref{Betti toric}, are all of the necessary ingredients for the proof of the following.
\begin{corollary}\label{reflexive h vector}
Let $\Delta$ be a smooth reflexive polytope of dimension $n$ with $h$-vector $\mathbf{h}=(h_0,\ldots,h_n)$. Define $\k0$ to be the great common divisor
of the affine lengths of its edges. Then
\begin{equation*}\label{corollary reflexive}
\sum_{j=0}^{\k0-1}y^j \quad \text{divides} \quad \sum_{j=0}^n h_{j}y^j.
\end{equation*}
\end{corollary}

Note that Corollary \ref{reflexive h vector} is a combinatorial result concerning reflexive polytope.
However its proof requires essentially two deep facts: the topological--combinatorial correspondence between monotone symplectic toric manifolds and smooth reflexive polytopes,
and the rigidity theorem of elliptic genera, which is the key ingredient for the proof of Corollary \ref{Betti toric}. Thus the following natural question arises:
\begin{question}\label{question cor}
Is there a different, possibly entirely combinatorial proof of Corollary \ref{reflexive h vector}?
\end{question}

\subsection{From the index to the Betti numbers: A closer look}\label{closer look}
In this subsection we prove Theorem \ref{main geometry}. Before specializing to the case in which the index is maximal,
we introduce some polynomials, together with their properties, that generalize the Hilbert polynomial
of an almost complex manifold.

Let $(\M^{2n}, \J)$ be a compact almost complex manifold of dimension $2n$ with non-zero first Chern class $c_1$. Let $N$ be a non-zero integer dividing $c_1$, namely $\frac{c_1}{N}\in H^2(\M;\Z)$.
Let $K=\wedge^nT^*$ and  $L=K^{\frac 1N}$, hence $c_1(L)=-\frac{c_1}{N}$. The \emph{Hilbert polynomial} of $(\M,\J)$ (see for instance \cite{S}) is the polynomial $H$ whose values at an integer $k$ are given by
$$
H(k):= \ind \left(L^k\right).
$$
Similarly, for every integer $k$ and $m>0$, we define
\begin{equation}\label{Hmk}
H_m(k):=\ind\left(L^k\otimes \wedge^m T^*\right).
\end{equation}
By the Atiyah-Singer Theorem it is easy to see that $H_m(k)$ depends on $k$ in a polynomial way, and the polynomial $H_m(x)$, with $x\in \C$ is defined to be the unique polynomial
whose values at an integer $k$ are given by \eqref{Hmk}.

The following proposition is a generalization of \cite[Proposition\ 4.1 (2)]{S}.

\begin{prop}\label{symmetries H}
Let $(\M^{2n}, \J, S^1)$ be a compact almost complex manifold which is acted on by a circle $S^1$ which preserves $\J$ and with discrete fixed point set $\M^{S^1}$.
Assume that the first Chern class $c_1$ is non-zero. Let
 $N$ be a non-zero integer and $L$  the line bundle $(\wedge^nT^*)^{\frac 1N}$. Then
 for every integer $k$ and $0\leq m \leq n$
\begin{equation}\label{eq symmetry}
\ind\left(L^{k}\otimes \wedge^mT^*\right)=(-1)^n\ind \left(L^{-k}
\otimes\wedge^{n-m}T^*\right)=(-1)^n\ind\left(L^{N-k}\otimes \wedge^m T\right).
\end{equation}
Hence for every $x\in \C$ we obtain
\begin{equation}\label{symmetry H}
H_m(x)=(-1)^n H_{n-m}(-x).
\end{equation}
\end{prop}
\begin{proof}
In order to prove \eqref{eq symmetry} we use an equivariant extension of the bundles involved. Indeed, \cite[Lemma 3.2]{Ha} implies that the line bundle $L$ has an equivariant extension, which is called $L_{S^1}$. The exterior powers of the tangent and cotangent bundles naturally inherit an $S^1$-action from that on $\M$.

As $c_1=-N\cdot c_1(L)$, the equivariant extensions of the Chern classes satisfy $c_1^{S^1}=-N\cdot c_1^{S^1}(L_{S^1})+a\,x$ for some $a\in \Z$, where $x$ is the degree two generator of $H^*_{S^1}(pt;\Z)$. We want to prove that, for the sake of this proof, the constant $a$ can be set to zero. Indeed first of all observe that the equivariant extension $L_{S^1}$ is not unique, as $L_{S^1}\otimes L_b$ is another such extension, where $L_b$ is the trivial line bundle
endowed with fiber-wise action given by $\lambda \cdot z = \lambda^b z$, for $b\in \Z$. At the level of Chern classes this means that, after choosing an equivariant extension $L_{S^1}$
with Chern class $c_1^{S^1}(L_{S^1})$, the Chern classes of all the other possible equivariant extensions are given by $c_1^{S^1}(L_{S^1})+b\,x$, with $b\in \Z$. Therefore
\begin{equation}\label{c1S1}
c_1^{S^1}=-N\left(c_1^{S^1}(L_{S^1})+b\,x\right)+a\,x.
\end{equation}
 By precomposing the $S^1$-action with the map from $S^1$ to $S^1$ that sends $\lambda$ to $\lambda^N$ we obtain a new action where all the weights at the fixed points get multiplied by $N$. Restricting \eqref{c1S1} to the fixed points, we obtain $\sum_{j=1}^n w_j(P)\,x=-N(c_1^{S^1}(L_{S^1})_{|_P}+b\,x)+a\,x$, where now $a$ is divisible by $N$. Hence we can choose
 a different equivariant extension of $L_{S^1}$ with $b=\frac{a}{N}$. Observe that the operations of choosing a different equivariant extension and of lifting the action change the equivariant index, but not the non equivariant one. Thus in what follows we can assume that
 \begin{equation}\label{relation c1 and L}
 \sum_{j=1}^n w_j(P)\,x=-N\cdot c_1^{S^1}(L_{S^1})|_{_P}.
 \end{equation}

Denote the sum of the weights $w_1(P)+\ldots+w_n(P)$ at a fixed point $P$ by $W(P)$. Then the localization theorem in equivariant K-theory \eqref{index K} implies that the equivariant index is equal
to
$$
\ind_{S^1}\left(L^{k}\otimes \wedge^mT^*\right)=\sum_{P\in \M^{S^1}}
\dfrac{t^{-k\frac{W(P)}{N}}e_m\left(t^{-w_1(P)}, \ldots,
t^{-w_n(P)}\right)}{\prod_{j=1}^n\left(1-t^{-w_j(P)}\right)},
$$
where $e_j$ denotes the $j$-th elementary symmetric polynomial in
$n$-variables.
Let $\widetilde{S}^1$ denote the circle $S^1$ with reversed orientation.
Then $\widetilde{S}^1$ acts on $M$ with weights at each fixed point
$P$ given by $-w_1(P),\ldots,-w_n(P)$. Hence
\begin{align*}
\ind_{\widetilde{S}^1}\left(L^{k}\otimes \wedge^mT^*\right)&=\sum_{P\in \M^{S^1}}
\dfrac{t^{k\frac{W(P)}{N}}e_m\left(t^{w_1(P)}, \ldots,
t^{w_n(P)}\right)}{\prod_{j=1}^n\left(1-t^{w_j(P)}\right)}\\
&=(-1)^n\sum_{P\in \M^{S^1}} \dfrac{t^{k\frac{W(P)}{N}}e_m\left(t^{w_1(P)},
\ldots,
t^{w_n(P)}\right)}{t^{W(P)}\prod_{j=1}^n(1-t^{-w_j(P)})}\\
&=(-1)^n\sum_{P\in \M^{S^1}} \dfrac{t^{(k-N)\frac{W(P)}{N}}e_m\left(t^{w_1(P)}, \ldots, t^{w_n(P)}\right)}{\prod_{j=1}^n\left(1-t^{-w_j(P)}\right)}\numberthis \label{sym2}\\
&=(-1)^n\ind_{S^1}\left(L^{N-k}\otimes \wedge^m T\right).
\end{align*}
Moreover, noting that $t^{-W(P)}e_m(t^{w_1(P)},\ldots,t^{w_n(P)})=e_{n-m}(t^{-w_1(P)},\ldots,t^{-w_n(P)})$ the expression in \eqref{sym2} also equals
\begin{align*}
&(-1)^n\sum_{P\in \M^{S^1}}
\dfrac{t^{k\frac{W(P)}{N}}e_{n-m}\left(t^{-w_1(P)}, \ldots,
t^{-w_n(P)}\right)}{\prod_{j=1}^n\left(1-t^{-w_j(P)}\right)}=(-1)^n\ind_{S^1}\left(L^{-k}
\otimes\wedge^{n-m}T^*\right).\\
\end{align*}
As observed in \eqref{index} the equalities in \eqref{eq symmetry} are obtained by forgetting the circle actions, namely by taking $t\to 1$.

>From the definition of $H_m(x)$ and \eqref{eq symmetry} it follows that $H_m(k)=(-1)^n H_{n-m}(-k)$ for every integer $k$. This implies \eqref{symmetry H}, as the polynomial
$H_m(x)-(-1)^n H_{n-m}(-x)$ has infinitely many zeros.
\end{proof}

As an example we compute these polynomials for the complex projective space.

\begin{prop}\label{projective space H}
The Hilbert polynomial $H_m(x)$ of the projective space
$\mathbb{CP}^n$ is given by
    $$
    H_m(x)=\dfrac{(-1)^n}{m!(n-m)!}(x-1)\cdot \ldots \cdot (x-(n-m))\cdot(x+1)\cdot \ldots \cdot(x+m).
    $$
    In particular,
    $$
    H_m(0)=\ind(\wedge^mT^*)=(-1)^m
    $$
    and
    $$
    \chi_y=1-y+y^2+\ldots+(-1)^ny^n.
    $$
\end{prop}

\begin{proof}
Consider the $S^1$-action on $\mathbb{CP}^n$ given by
$$
\lambda \cdot \left[z_0: z_1: \ldots: z_n\right]=\left[z_0: \lambda^{w_1}z_1:
\ldots: \lambda^{w_n}z_n\right]\,,
$$
where $w_1, \ldots, w_n$ are distinct, non-zero integers.

This $S^1$-action is the restriction to a circle of the standard
toric action of the $n$-dimensional torus $T^n$ on
$\mathbb{CP}^n\,.$ The circle action has $P_0:=\{[1:0:\ldots:0],
P_1:=[0:1:\ldots:0], \ldots, P_n:=[0:0:\ldots:1]\}$ as a set of
fixed points. The weights of the circle action are given at $P_0$ by
$\{w_k\}_{k=1}^n$ and at $P_j$ by
$\{-w_j+(1-\delta_{jk})w_k\}_{k=1}^n$ for $j\in\{1, \ldots, n\}.$

Let $L$ be the line bundle such that $c_1(L)=-\frac{c_1}{n+1}.$
We argue as in the proof of Proposition \ref{symmetries H} and consider
an equivariant extension of $L$ satisfying \eqref{relation c1 and L}.
Thus in this case it is easy to check that
$L^k_{S^1}(P_0)=t^{-k\frac{w_1+\ldots+w_n}{n+1}}$ and
$L_{S^1}^k(P_j)=t^{-k\frac{w_1+\ldots+w_n}{n+1}}t^{k w_j}$, for all $j\in\{1,\ldots,n\}$.
Since the integers $w_1,\ldots,w_n$ can be chosen arbitrarily as long as they are
all distinct and non-zero (as otherwise the action would not have isolated fixed points)
we choose them such that $w_1+\cdots +w_n=0$.
Then the localization formula for the computation of the index yields
\begin{align*}
\ind_{S^1}\left(L_{S^1}^k\otimes \wedge^mT^*\right)=&\dfrac{e_m(t^{-w_1}, \ldots,
t^{-w_n})}{\prod_{j=1}^n(1-t^{-w_j})}+\sum_{j=1}^n\dfrac{t^{kw_j}e_m(t^{w_j-w_1},\ldots,
t^{w_j},\ldots, t^{w_j-w_n})}{(1-t^{w_j})\prod_{r=1,r\ne
j}^n(1-t^{w_j-w_r})}\,.\\
\end{align*}
Now we would like to express the equivariant index of $L_{S^1}^k\otimes
\wedge^mT^*$ in terms of the equivariant index of $L_{S^1}^k,$ the latter
given by the formula
\begin{align*}
\ind_{S^1}\left(L_{S^1}^k\right)&=\dfrac{1}{\prod_{j=1}^n(1-t^{-w_j})}+\sum_{j=1}^n\dfrac{t^{kw_j}}{(1-t^{w_j})\prod_{r=1,r\ne
j}^n(1-t^{w_j-w_r})}\,.
\end{align*}
>From the identity below between elementary symmetric polynomials
\begin{multline*}
e_m(t^{w_j-w_1},\ldots, t^{w_j},\ldots,
t^{w_j-w_n})=t^{mw_j}e_m(t^{-w_1},\ldots,t^{-w_j},\ldots,
t^{-w_n})\\+(-1)^{m+1}(1-t^{-w_j})\sum_{p=0}^{m-1}(-1)^pt^{w_j(p+1)}e_{p}(t^{-w_1},\ldots,t^{-w_j},\ldots,
t^{-w_n}),
\end{multline*}
it follows that
\begin{align*} &\ind_{S^1}\left(L_{S^1}^k\otimes
\wedge^mT^*\right)=e_m(t^{-w_1}, \ldots,
t^{-w_n})\ind_{S^1}\left(L_{S^1}^{k+m}\right)\\&+(-1)^{m+1}\sum_{p=0}^{m-1}(-1)^pe_{p}\left(t^{-w_1},\ldots,
t^{-w_n}\right)\left(\ind_{S^1}\left(L_{S^1}^{k+p+1}\right)-\ind_{S^1}\left(L_{S^1}^{k+p}\right)\right)\,.
\end{align*}
By taking $t\to 1,$ we obtain the equality
\begin{multline*}
\ind\left(L^k\otimes
\wedge^mT^*\right)\\
=\binom{n}{m}\ind\left(L^{k+m}\right)+(-1)^{m+1}\sum_{p=0}^{m-1}(-1)^p\binom{n}{p}\left(\ind\left(L^{k+p+1}\right)-\ind\left(L^{k+p}\right)\right)\,.
\end{multline*}
Recall that the Hilbert polynomial $H_0(x)$ of the projective space
$\mathbb{CP}^n$ is given by
$$
H_0(x)=\dfrac{(-1)^n}{n!}(x-1)\cdot\ldots \cdot(x-n)\,.
$$
We claim that for all integers $k \neq 0$ and all integers $0 \leq m
\leq n,$ the Hilbert polynomial $H_0(x)$ satisfies the formula
\begin{equation}\label{claim}
    \binom{n}{m}\frac{m}{k}H_{0}(k+m) =
    (-1)^{m+1}\sum_{p=0}^{m-1}(-1)^{p}\binom{n}{p}n\frac{H_{0}(k+p+1)}{(k+p)}\,.
\end{equation}
The statement follows from the formula as it implies that
\begin{align*}
&\ind\left(L^k\otimes \wedge^mT^*\right)\\
&= \binom{n}{m}H_0(k+m)+
(-1)^{m+1}\sum_{p=0}^{m-1}(-1)^{p}\binom{n}{p}\big(H_{0}(k+p+1)-H_{0}(k+p)
\big)\\
&=\binom{n}{m}H_0(k+m)+(-1)^{m+1}\sum_{p=0}^{m-1}(-1)^{p}\binom{n}{p}n\frac{H_{0}(k+p+1)}{(k+p)}\\
&=\binom{n}{m}H_0(k+m)+\dfrac{m}{k}\binom{n}{m}H_0(k+m)=\binom{n}{m}\dfrac{H_0(k+m)(k+m)}{k}\\
&=\dfrac{(-1)^n}{m!(n-m)!}(k-1)\cdot \ldots \cdot
(k-(n-m))(k+1)\cdot \ldots (k+m)\,.
\end{align*}
We conclude the proof by proving the claim for fixed $k$ by induction on $m$. For
$m = 1$ the formula is true. We next show that both sides satisfy
the same recursion as $m \mapsto m + 1$. Write $h_{L}(m)$ and
$h_{R}(m)$ for the left-hand side and the right-hand side of the
formula in \ref{claim}, respectively. First, we have the recursion
    \[
    h_{R}(m+1) - h_{R}(m) = (-1)^{m}\binom{n}{m}n\frac{H_{0}(k+m+1)}{k+
    m}.
    \]
    Furthermore, we compute
    \begin{align*}
    h_{L}(m&+1) - h_{L}(m) = (-1)^{m}\binom{n}{m+1}\frac{m+1}{k}H_{0}(k+m+1) - (-1)^{m+1}\binom{n}{m}\frac{m}{k}H_{0}(k+m) \\
    &=  (-1)^{m}\binom{n}{m+1}\frac{m+1}{k}H_{0}(k+m+1) - (-1)^{m+1}\binom{n}{m}\frac{m}{k}\frac{H_{0}(k+m+1)}{k+m}(k+m-n) \\
    &=(-1)^{m}\binom{n}{m}n\frac{H_{0}(k+m+1)}{k+
    m}.
    \end{align*}
Thus $h_{L}(m)$ and $h_{R}(m)$ satisfy the same recursion if
$m\mapsto m+1$, hence the formula follows by induction and we finish
the proof of the claim.
\end{proof}

\begin{rmk}In the figure below, the black dots represent the position on
the real line of the zeroes  of the Hilbert polynomials $H_m(x)$ for
$m\in\{0,1,\ldots, n-1,n\}.$ At each stage, starting from $m=0,$ the
zeroes of the polynomials are shifted one unit to the left but they
intriguingly jump one more unit left when they reach the origin.

\begin{figure}[!ht]
\centering
\begin{tikzpicture}[scale=1.2][transform shape]
  \node [scale=0.7,above] at (-4.7,0.3) {$m=0$};
  \node[draw,fill=white,shape=circle,scale=0.4] (a) at (-4.5,0) {};
  \node[draw,fill=white,shape=circle,scale=0.4] (b) at (-3.5,0) {};
  \node[scale=0.5] (c) at (-3,0) {};
  \node[scale=0.5] (d) at (-2.5,0) {};
  \node[draw,fill=white,shape=circle,scale=0.4] (e) at (-2,0) {};
  \node[draw,fill=white,shape=circle,scale=0.4] (f) at (-1,0) {};
  \node[draw,fill=white,shape=circle,scale=0.4] (g) at (0,0) {};
  \node[draw,fill=black,shape=circle,scale=0.4] (h) at (1,0) {};
  \node[draw,fill=black,shape=circle,scale=0.4] (i) at (2,0) {};
  \node[scale=0.5] (j) at (2.5,0) {};
  \node[scale=0.5] (k) at (3,0) {};
  \node[draw,fill=black,shape=circle,scale=0.4] (l) at (3.5,0) {};
  \node[draw,fill=black,shape=circle,scale=0.4] (m) at (4.5,0) {};
  \draw  (a) -- (b)
         (b) -- (c)
         (d) -- (e)
         (e) -- (f)
         (f) -- (g)
         (g) -- (h)
         (h) -- (i)
         (i) -- (j)
         (k) -- (l)
         (l) -- (m);
\node[draw,shape=circle,scale=0.05,fill=black]  at (2.6,0) {};
\node[draw,shape=circle,scale=0.05,fill=black]  at (2.75,0) {};
\node[draw,shape=circle,scale=0.05,fill=black]  at (2.9,0) {};
\node[draw,shape=circle,scale=0.05,fill=black]  at (-2.6,0) {};
\node[draw,shape=circle,scale=0.05,fill=black]  at (-2.75,0) {};
\node[draw,shape=circle,scale=0.05,fill=black]  at (-2.9,0) {};
 \node [scale=0.5,below] at (-1,-0.2) {-1};
 \node [scale=0.5,below] at (-2,-0.2) {-2};
 \node [scale=0.5,below] at (-3.5,-0.2) {$-(n-1)$};
 \node [scale=0.5,below] at (-4.5,-0.2) {$-n$};
 \node [scale=0.5,below] at (0,-0.2) {0};
 \node [scale=0.5,below] at (1,-0.2) {1};
 \node [scale=0.5,below] at (2,-0.2) {2};
 \node [scale=0.5,below] at (3.5,-0.2) {$n-1$};
 \node [scale=0.5,below] at (4.5,-0.2) {$n$};

  \node [scale=0.7,above] at (-4.7,-0.9) {$m=1$};
  \node[draw,fill=white,shape=circle,scale=0.4] (a2) at (-4.5,-1.2){};
  \node[draw,fill=white,shape=circle,scale=0.4] (b2) at (-3.5,-1.2) {};
  \node[scale=0.5] (c2) at (-3,-1.2) {};
  \node[scale=0.5] (d2) at (-2.5,-1.2) {};
  \node[draw,fill=white,shape=circle,scale=0.4] (e2) at (-2,-1.2) {};
  \node[draw,fill=black,shape=circle,scale=0.4] (f2) at (-1,-1.2) {};
  \node[draw,fill=white,shape=circle,scale=0.4] (g2) at (0,-1.2) {};
  \node[draw,fill=black,shape=circle,scale=0.4] (h2) at (1,-1.2) {};
  \node[draw,fill=black,shape=circle,scale=0.4] (i2) at (2,-1.2) {};
  \node[scale=0.5] (j2) at (2.5,-1.2) {};
  \node[scale=0.5] (k2) at (3,-1.2) {};
  \node[draw,fill=black,shape=circle,scale=0.4] (l2) at (3.5,-1.2) {};
  \node[draw,fill=white,shape=circle,scale=0.4] (m2) at (4.5,-1.2) {};
  \draw  (a2) -- (b2)
         (b2) -- (c2)
         (d2) -- (e2)
         (e2) -- (f2)
         (f2) -- (g2)
         (g2) -- (h2)
         (h2) -- (i2)
         (i2) -- (j2)
         (k2) -- (l2)
         (l2) -- (m2);
\node[draw,shape=circle,scale=0.05,fill=black]  at (2.6,-1.2) {};
\node[draw,shape=circle,scale=0.05,fill=black]  at (2.75,-1.2) {};
\node[draw,shape=circle,scale=0.05,fill=black]  at (2.9,-1.2) {};
\node[draw,shape=circle,scale=0.05,fill=black]  at (-2.6,-1.2) {};
\node[draw,shape=circle,scale=0.05,fill=black]  at (-2.75,-1.2) {};
\node[draw,shape=circle,scale=0.05,fill=black]  at (-2.9,-1.2) {};
 \node [scale=0.5,below] at (-1,-1.4) {-1};
 \node [scale=0.5,below] at (-2,-1.4) {-2};
 \node [scale=0.5,below] at (-3.5,-1.4) {$-(n-1)$};
 \node [scale=0.5,below] at (-4.5,-1.4) {$-n$};
 \node [scale=0.5,below] at (0,-1.4) {0};
 \node [scale=0.5,below] at (1,-1.4) {1};
 \node [scale=0.5,below] at (2,-1.4) {2};
 \node [scale=0.5,below] at (3.5,-1.4) {$n-1$};
 \node [scale=0.5,below] at (4.5,-1.4) {$n$};

\node[draw,shape=circle,scale=0.05,fill=black]  at (-0.3,-2.1) {};
\node[draw,shape=circle,scale=0.05,fill=black]  at (0,-2.1) {};
\node[draw,shape=circle,scale=0.05,fill=black]  at (0.3,-2.1) {};

\node [scale=0.7,above] at (-4.5,-2.8) {$m=n-1$};
  \node[draw,fill=white,shape=circle,scale=0.4] (a3) at (-4.5,-3.1){};
  \node[draw,fill=black,shape=circle,scale=0.4] (b3) at (-3.5,-3.1) {};
  \node[scale=0.5] (c3) at (-3,-3.1) {};
  \node[scale=0.5] (d3) at (-2.5,-3.1) {};
  \node[draw,fill=black,shape=circle,scale=0.4] (e3) at (-2,-3.1) {};
  \node[draw,fill=black,shape=circle,scale=0.4] (f3) at (-1,-3.1) {};
  \node[draw,fill=white,shape=circle,scale=0.4] (g3) at (0,-3.1) {};
  \node[draw,fill=black,shape=circle,scale=0.4] (h3) at (1,-3.1) {};
  \node[draw,fill=white,shape=circle,scale=0.4] (i3) at (2,-3.1) {};
  \node[scale=0.5] (j3) at (2.5,-3.1) {};
  \node[scale=0.5] (k3) at (3,-3.1) {};
  \node[draw,fill=white,shape=circle,scale=0.4] (l3) at (3.5,-3.1) {};
  \node[draw,fill=white,shape=circle,scale=0.4] (m3) at (4.5,-3.1) {};
  \draw  (a3) -- (b3)
         (b3) -- (c3)
         (d3) -- (e3)
         (e3) -- (f3)
         (f3) -- (g3)
         (g3) -- (h3)
         (h3) -- (i3)
         (i3) -- (j3)
         (k3) -- (l3)
         (l3) -- (m3);
\node[draw,shape=circle,scale=0.05,fill=black]  at (2.6,-3.1) {};
\node[draw,shape=circle,scale=0.05,fill=black]  at (2.75,-3.1) {};
\node[draw,shape=circle,scale=0.05,fill=black]  at (2.9,-3.1) {};
\node[draw,shape=circle,scale=0.05,fill=black]  at (-2.6,-3.1) {};
\node[draw,shape=circle,scale=0.05,fill=black]  at (-2.75,-3.1) {};
\node[draw,shape=circle,scale=0.05,fill=black]  at (-2.9,-3.1) {};
 \node [scale=0.5,below] at (-1,-3.3) {-1};
 \node [scale=0.5,below] at (-2,-3.3) {-2};
 \node [scale=0.5,below] at (-3.5,-3.3) {$-(n-1)$};
 \node [scale=0.5,below] at (-4.5,-3.3) {$-n$};
 \node [scale=0.5,below] at (0,-3.3) {0};
 \node [scale=0.5,below] at (1,-3.3) {1};
 \node [scale=0.5,below] at (2,-3.3) {2};
 \node [scale=0.5,below] at (3.5,-3.3) {$n-1$};
 \node [scale=0.5,below] at (4.5,-3.3) {$n$};

  \node [scale=0.7,above] at (-4.7,-4) {$m=n$};
  \node[draw,fill=black,shape=circle,scale=0.4] (a4) at (-4.5,-4.3){};
  \node[draw,fill=black,shape=circle,scale=0.4] (b4) at (-3.5,-4.3) {};
  \node[scale=0.5] (c4) at (-3,-4.3) {};
  \node[scale=0.5] (d4) at (-2.5,-4.3) {};
  \node[draw,fill=black,shape=circle,scale=0.4] (e4) at (-2,-4.3) {};
  \node[draw,fill=black,shape=circle,scale=0.4] (f4) at (-1,-4.3) {};
  \node[draw,fill=white,shape=circle,scale=0.4] (g4) at (0,-4.3) {};
  \node[draw,fill=white,shape=circle,scale=0.4] (h4) at (1,-4.3) {};
  \node[draw,fill=white,shape=circle,scale=0.4] (i4) at (2,-4.3) {};
  \node[scale=0.5] (j4) at (2.5,-4.3) {};
  \node[scale=0.5] (k4) at (3,-4.3) {};
  \node[draw,fill=white,shape=circle,scale=0.4] (l4) at (3.5,-4.3) {};
  \node[draw,fill=white,shape=circle,scale=0.4] (m4) at (4.5,-4.3) {};
  \draw  (a4) -- (b4)
         (b4) -- (c4)
         (d4) -- (e4)
         (e4) -- (f4)
         (f4) -- (g4)
         (g4) -- (h4)
         (h4) -- (i4)
         (i4) -- (j4)
         (k4) -- (l4)
         (l4) -- (m4);
\node[draw,shape=circle,scale=0.05,fill=black]  at (2.6,-4.3) {};
\node[draw,shape=circle,scale=0.05,fill=black]  at (2.75,-4.3) {};
\node[draw,shape=circle,scale=0.05,fill=black]  at (2.9,-4.3) {};
\node[draw,shape=circle,scale=0.05,fill=black]  at (-2.6,-4.3) {};
\node[draw,shape=circle,scale=0.05,fill=black]  at (-2.75,-4.3) {};
\node[draw,shape=circle,scale=0.05,fill=black]  at (-2.9,-4.3) {};
 \node [scale=0.5,below] at (-1,-4.5) {-1};
 \node [scale=0.5,below] at (-2,-4.5) {-2};
 \node [scale=0.5,below] at (-3.5,-4.5) {$-(n-1)$};
 \node [scale=0.5,below] at (-4.5,-4.5) {$-n$};
 \node [scale=0.5,below] at (0,-4.5) {0};
 \node [scale=0.5,below] at (1,-4.5) {1};
 \node [scale=0.5,below] at (2,-4.5) {2};
 \node [scale=0.5,below] at (3.5,-4.5) {$n-1$};
 \node [scale=0.5,below] at (4.5,-4.5) {$n$};
\end{tikzpicture}
\end{figure}

\end{rmk}

In order to introduce the second important property of the polynomials $H_m$, we first recall the following:
For a compact almost complex manifold equipped with a circle action and isolated fixed points, it is well-known (see for instance \cite{HBJ,Li}) that
$\ind(\wedge^m T^*)=(-1)^m N_m$, where $N_m$ denotes the number of fixed points with $m$ negative weights.
Since $H_m(0)$ is by definition $\ind(\wedge^m T^*)$,
one has that
$$
\sharp{\text{fixed points}}=\sum_{m=0}^n N_m = \sum_{m=0}^n (-1)^m\,H_m(0).
$$
The next proposition generalizes the above equation.
\begin{prop}\label{fixed points and polynomials}
Let $(\M^{2n}, \J, S^1)$ be a compact almost complex manifold which is acted on by a circle $S^1$ which preserves $\J$ and with discrete fixed point set $\M^{S^1}$.
Then, for every $x\in \C$,
\begin{align*}
\sharp{\text{fixed points}}=\sum_{m=0}^n (-1)^m\,H_m(x).
\end{align*}
\end{prop}
\begin{proof}
Following the first part of the proof of Proposition \ref{symmetries H}, we consider an equivariant extension of the bundle $L$ with respect to
the new action of $S^1$ such that the restriction of $L$ to the fiber over a fixed point $P$ is the representation $t^{-\frac{w_1(P)+\cdots + w_n(P)}{N}}=t^{-\frac{W(P)}{N}}$. Then,
for every integer $k$, we have
\begin{align*}
\sum_{P\in \M^{S^1}} t^{-k\frac{W(P)}{N}}&=\sum_{P\in \M^{S^1}}
\dfrac{t^{-k\frac{W(P)}{N}}\left(1-t^{-w_1(P)}\right)\cdot\ldots\cdot\left(1-t^{-w_n(P)}\right)}{
\left(1-t^{-w_1(P)}\right)\cdot\ldots\cdot\left(1-t^{-w_n(P)}\right)}\\
&=\sum_{P\in \M^{S^1}}
\dfrac{t^{-k\frac{W(P)}{N}}\sum_{m=0}^n(-1)^me_m\left(t^{-w_1(P)},
\ldots,
t^{-w_n(P)}\right)}{\left(1-t^{-w_1(P)}\right)\cdot\ldots\cdot\left(1-t^{-w_n(P)}\right)}\\
&=\sum_{m=0}^n (-1)^m\ind_{S^1}\left(L^k\otimes \wedge^mT^*\right).
\end{align*}
By taking $t\to 1,$ we obtain
\begin{align*}
\sharp{\text{fixed points}}=\sum_{m=0}^n (-1)^m\ind\left(L^k\otimes
\wedge^mT^*\right)=\sum_{m=0}^n(-1)^m\,H_m(k).
\end{align*}
As the result holds for every integer $k,$ it holds for every $x.$
\end{proof}

We are now ready to specialize to the case in which the index is maximal.
First of all we have the following
\begin{lemma}\label{Hn}
Let $(\M^{2n}, \J, S^1)$ be a compact almost complex manifold which is acted on by a circle $S^1$ which preserves $\J$ and with discrete fixed point set $\M^{S^1}$.
Let $N_0$ be the number of fixed points with no negative weights.
Assume that the index $\k0$ is $n+1$. Then
$$
H_n(x)=(-1)^n\dfrac{N_0}{n!}(x+1)\cdot\ldots\cdot (x+n).
$$
\end{lemma}
\begin{proof}
From \cite[Proposition\ 5.1]{S} we know that the polynomial $H_0$ is known in this case, and it is given by
$$
H_0(x)=(-1)^n\frac{N_0}{n!}(x-1)\cdots(x-n).
$$
Note that the polynomial $\operatorname{H}(x)$ in \cite{S} is exactly $H_0(-x)$.
Then equation \eqref{symmetry H} gives the desired conclusion.
\end{proof}
The next theorem is the key ingredient for the proof of Theorem \ref{main geometry}.
\begin{thm}\label{main 2}
Let $(\M, \J, S^1)$ be a compact almost complex manifold of dimension $2n$ which is acted on by a circle $S^1$ which preserves $\J$ and with discrete fixed point set.
Let $N_0$ be the number of fixed points with $0$ negative weights.
Assume that the index $\k0$ is $n+1$. If the elliptic genus $\varphi_{n+1}(\M)$ vanishes, then
$$
\sharp{\text{fixed points}}=N_0(n+1).
$$
\end{thm}
In order to prove Theorem \ref{main 2} we need to analyze the first terms of the Fourier expansions of the elliptic genus of level $N = n+1$ at the cusps. First note that by \eqref{eq elliptic genus cusps} the vanishing of the elliptic genus $\varphi_{N}(\M)$ implies its vanishing at all cusps, i.e., the vanishing of $\varphi_{N,(k,\ell)}(\M)$ for every primitive $N$-division point $w = \frac{k}{N}\tau+\frac{\ell}{N}$. In particular, we also have $\widetilde{\varphi}_{N,(k,\ell)}(\M)(q^{N}) = 0$ for the normalized elliptic genus defined in \eqref{eq normalized elliptic genus}. On the other hand, from \eqref{infinite tensor}, it is easy to see that the first terms of $\widetilde{\varphi}_{N,(k,\ell)}(\M)(q^{N})$ in terms of powers of $q$ are given by
\begin{multline} \label{expansion}
\widetilde{\varphi}_{N, (k,
\ell)}(\M)\left(q^{N}\right)=\ind\left(L^k\right)-\ind\left(L^k\otimes
T^*\right)\zeta_{N}^{\ell}q^k+\ldots+(-1)^m\ind\left(L^k\otimes
\wedge^mT^*\right)\zeta_{N}^{\ell m}q^{k m}\\-\ind\left(L^k\otimes
T\right)\zeta_{N}^{-\ell}q^{N-k}+\ldots+(-1)^r\ind\left(L^k\otimes
\wedge^r T\right)\zeta_{N}^{-\ell r}q^{(N-k)r} + \cdots,
\end{multline}
\noindent where $m$ and $r$ are chosen so that $km$ and $(N-k)r <N$,
and what is left contains powers of $q$ with exponent greater than
$\max\{km,(N-k)r\}$. The next proposition analyses this expansion
when $N=n+1$.

\begin{prop}\label{ellipticgenus0}
Let $(\M,\J)$ be a compact almost complex manifold of dimension
$2n$, with $n \geq 2$. Assume that the index $\k0$ is $n+1.$ If
$\varphi_{n+1}(\M)=0$ then
$$
H_m(1)=\ind\left(L\otimes \wedge^mT^*\right)=0
$$
for $0 \leq m \leq n-1.$
\end{prop}
\begin{proof}
As explained above, the vanishing of $\varphi_{N}(\M)$ implies
$\widetilde{\varphi}_{N,(k,\ell)}(\M)(q^{N})$ for every primitive
$N$-division point $\frac{k}{N}\tau + \frac{\ell}{N}$. When $N =
n+1$ and $k=1$ we can choose $m=n-1,r=1$ and obtain from
\eqref{expansion}
\begin{align*}
0 &= \widetilde{\varphi}_{n+1, (1, \ell)}(\M)\left(q^{n+1}\right) \\
&=\ind(L)-\ind(L\otimes
T^*)\zeta_{n+1}^\ell q^1+\ldots+(-1)^{n-1}\ind(L\otimes
\wedge^{n-1}T^*)\zeta_{n+1}^{\ell(n-1)}q^{(n-1)}\\
&\quad +(-1)^n\ind(L\otimes \wedge^nT^*)\zeta_{n+1}^{ln}q^n
-\ind(L\otimes T)\zeta_{n+1}^{-\ell}q^{n}+\text{higher terms.}
\end{align*}

Since for $n\geq 2$ the exponents of $q$ in the above expression are
all distinct, we obtain the desired claim.
\end{proof}

\begin{proof}[Proof of Theorem \ref{main 2}]
Combining Proposition \ref{fixed points and polynomials} for $x=1$ and Proposition \ref{ellipticgenus0} we have that the number of fixed points is $\sum_{m=0}^n (-1)^m H_m(1)=(-1)^n H_n(1)$ which, by Lemma \ref{Hn}, equals $N_0 (n+1)$.
\end{proof}

The next theorem combines results already known in the literature, and is the second key ingredient for the proof of Theorem \ref{main geometry}.

\begin{thm}\label{main 3}
Let $(\M,\omega,\psi)$ be a compact, connected symplectic manifold of dimension $2n$ which is acted on by a circle in a Hamiltonian way and isolated fixed points.
Let $\chi(\M)$ be its Euler characteristic and $\k0$ its index, and assume that $\chi(\M)=\k0=n+1$. Then $\M$ is complex cobordant and homotopy equivalent to $\C P^n$.
\end{thm}
\begin{proof}
The proof of this theorem combines results of Hattori \cite{Ha}, Tolman \cite{T}, and Charton \cite{C} in the following way.

 Since we are assuming $\chi(\M)=n+1$ and $\M$ is endowed with a Hamiltonian $S^1$-action with isolated fixed points,  the Betti numbers satisfy $b_{2j}(\M)=1$
for all $j\in\{0,\ldots,n\}$. As in particular $b_2(\M)=1$ we can rescale the symplectic form to be integral, namely $[\omega]\in H^2(\M;\Z)$ (or rather the image of this group
in $H^2(\M;\R)$). Consider the prequantization line bundle
$L$ whose first Chern class is $[\omega]$ and observe that it admits an equivariant extension $\LS$. Indeed, since the action is Hamiltonian, the symplectic form can be completed
to an equivariant symplectic form $[\omega + \psi]\in H^2_{S^1}(\M;\Z)$, condition that implies the existence of $\LS$ (see \cite[Theorem 1.1, Corollary 1.2]{HY}). The equivariance of the projection map
$\LS \to \M$ implies that over a fixed point $P$ the complex line $\LS(P)$, the restriction of $\LS$ at $P$, inherits an $S^1$-action; moreover the weight of the $S^1$-representation is exactly given
by $\psi(P)$. By \cite[Proposition 3.4]{T} we have that $\psi(P)\neq \psi(Q)$ for all pairs of distinct fixed points $P$ and $Q$. Finally, as $ \int_\M c_1(L)=\int_\M \omega^n \neq 0$,
 $L$ is quasi-ample in the sense of Hattori \cite[Section 3]{Ha}.
 Moreover, since $H^2(\M;\Z)$ is a one-dimensional lattice and the index of $\M$ is $n+1$, we can rescale further the symplectic form to satisfy $c_1 = (n+1)[\omega]$, and hence the equivariant extensions to satisfy $c_1^{S^1} = (n+1)[\omega + \psi] + a$ for some $a\in \Z$. Restricting the previous expression to the fixed points yields $\sum_{j=1}^n w_j(P)= (n+1)\psi(P)+a$
 for every $P\in \fixed$. We can conclude that $L$ satisfies what Hattori calls condition D (see \cite[page 447]{Ha}). In conclusion
we can apply \cite[Theorem 5.7]{Ha}, that asserts that the $S^1$-action at the fixed point set resembles that of the standard $S^1$-action on $\C P^n$ (see the proof of Proposition \ref{projective space H}), namely, there exist integers $w_1,\ldots,w_n$, such that the weights at the $n+1$ fixed points of the action on $\M$ are exactly those given by the standard linear action on $\C P^n$. Since all Chern numbers can be computed entirely from the weights at the fixed points (see \eqref{chern number weights}), we deduce that
all Chern numbers agree with those of $\C P^n$ with standard complex structure. This, in turns, implies that $\M$ is complex cobordant to $\C P^n$.

To deduce the existence of a homotopy equivalence we apply results of Tolman \cite{T} and Charton \cite{C}.
First of all, \cite[Corollary 3.19]{T} asserts that for all $j\in\{0,\ldots,n\}$,
the group $H^{2j}(\M;\Z)$ is generated by $c_1^{\,j}$ suitably rescaled, and the rescaling factor depends just on the weights at the fixed points. It follows that if the weights are standard --namely they agree with those of the standard action on $\C P^n$-- then the cohomology ring is isomorphic to that of $\C P^n$.
Given that the cohomology rings are isomorphic, the existence of the homotopy equivalence follows then from the results of Charton in \cite{C}. Indeed by Morse theory $\M$ is homotopy equivalent to a CW complex with exactly one cell of dimension $2j$, for all $j\in\{0,\ldots,n\}$, and the conclusion then follows from \cite[Theorem 5.2]{C}.
\end{proof}

We are now ready to give a proof of Theorem \ref{main geometry}.
\begin{proof}[Proof of Theorem \ref{main geometry}]\label{proof main geometry}
If $\M$ is complex cobordant to $\C P^n$, then their elliptic genera are the same; indeed (elliptic) genera only depend on their generating function and the Chern numbers of the
manifold which, for complex cobordant manifolds, are the same. However the elliptic genus of level $n+1$ of $\C P^n$ vanishes, as it is proved\footnote{An alternative proof of this
fact is given by Proposition \ref{toric vanishing}, as $\C P^n$ admits a toric action.} in \cite{HBJ}.

Conversely, suppose that $(\M,\omega,\psi)$ is a compact, connected symplectic manifold of dimension $2n$ and index $\k0=n+1$ which is acted on effectively by a circle in a Hamiltonian way, with moment
map $\psi\colon \M\to \R$ and isolated fixed points. Then $\M$ can be endowed with an almost complex structure $\J$ compatible with $\omega$ and invariant under the $S^1$-action.
Hence, if we assume that the elliptic genus of level $n+1$ vanishes, then $(\M, \J, S^1)$ satisfies the hypotheses of Theorem \ref{main 2}.
Since the action is Hamiltonian and the manifold connected, the number of fixed points with no negative weights is one. Hence the number of fixed points $|\M^{S^1}|$ is exactly
$n+1$. However it is well-known that, if the action has discrete fixed point set, then the Euler characteristic $\chi(\M)$ is precisely $|\M^{S^1}|$.
Thus we can apply Theorem \ref{main 3} to conclude the proof.
\end{proof}

\begin{corollary}
Under the same assumptions as Theorem \ref{main 3}, the $m$-th
Hilbert polynomial $H_m(x)$ of $(\M, \omega, \psi) $ is given by
    $$
    H_m(x)=\dfrac{(-1)^n}{m!(n-m)!}(x-1)\cdot \ldots \cdot (x-(n-m))\cdot(x+1)\cdot \ldots
    \cdot (x+m)\,.
    $$
\end{corollary}
\begin{proof}
As argued in the proof of Theorem \ref{main 3}, the
complex projective space $\mathbb{CP}^n$ and $\M$ have the same Chern
numbers. The statement follows because the Hilbert
polynomials of an almost complex manifold are totally determined by
their Chern numbers.
\end{proof}

\section{From geometry to number theory: Relations of Eisenstein series}\label{section number theory proofs}

    In this section we explain how the rigidity theorem for the elliptic genus of level $N$ yields curious identities for products of Eisenstein series for $\Gamma_{1}(N)$. Thereby, we prove Theorem~\ref{main number theory}. As an example, we also give an infinite family of relations coming from the rigidity of $\C P^{n}$.

    Let $(\M,\J)$ be a compact, connected, almost complex manifold of dimension $2n$ whose first Chern class is divisible by $N$. Suppose that a circle $S^{1}$ acts effectively on $\M$, and that the fixed point set $\M^{S^{1}}$ is non-empty and discrete. Let $w_{1}(P),\dots,w_{n}(P) \in \Z$ denote the weights at a fixed point $P \in \M^{S^{1}}$. Our assumption on the fixed point set of the $S^{1}$-action implies that all the weights are non-zero,
    compare to Section~\ref{generalities s1 action}. We now come to the proof of Theorem~\ref{main number theory}.

    \begin{proof}[Proof of Theorem~\ref{main number theory}] Consider the Laurent series
    \[
    F_{N}(x) := \frac{\mathcal Q_{N}(x)}{x},
    \]
    where $\mathcal Q_{N}(x)$ is the power series defined in \eqref{eq Qtau} for the elliptic genus $\varphi_{N}(\M)$ of level $N$.
    The equivariant elliptic genus of level $N$ associated to $\M$ is given for $1 \neq t = e^{2\pi i z} \in S^{1}$ as
    \[
    \varphi_{N}(\M,t) = \sum_{P\in \M^{S^{1}}}F_{N}\left(2\pi i w_{1}(P)z\right)\cdots F_{N}\left(2\pi i w_{n}(P)z\right),
    \]
   compare to Section~\ref{section rigidity}. The rigidity theorem, Theorem~\ref{rigidity}, asserts that $\varphi_{N}(\M,t)$ is actually independent of $t \in S^{1}$. In particular, this implies that the non-constant
   Taylor coefficients of $\varphi_{N}(\M,t)$ around $z = 0$ vanish.
    A short computation
     shows that the Taylor coefficient at $z^{k-n}$ is explicitly given by
    \begin{align}\label{eq taylor coefficient}
    \sum_{I \in P_{n}(k)}G_{I,N}(\tau)\sum_{P \in \M^{S^{1}}}\frac{m_{I}\left(w_{1}(P),\dots,w_{n}(P)\right)}{w_{1}(P)\cdots w_{n}(P)},
    \end{align}
    where  $m_{I}$ is the usual monomial symmetric polynomial. Here we also use Lemma~\ref{lemma eisenstein} which states that the coefficients $a_{k}$ of $\mathcal Q_{N}(x)$ are given by the Eisenstein series $G_{k,N}$.
    By the rigidity theorem, Theorem~\ref{rigidity}, the expression in \eqref{eq taylor coefficient} vanishes for $k > n$, which finishes the proof of Theorem~\ref{main number theory}.
    \end{proof}

For every $I\in P_n(k)$ the coefficient of $G_{I,N}(\tau)$ has the following geometric interpretation.
As $m_I(x_1,\ldots,x_n)$ is symmetric in $x_1,\ldots,x_n$, there exists a polynomial $Q_I(y_1,\ldots,y_n)$ such that
\begin{equation}\label{Q I}
Q_I(e_1,\ldots,e_n)=m_I,
\end{equation}
 where $e_j(x_1,\ldots,x_n)$ is the elementary symmetric polynomial of degree $j$.
By \eqref{restriction chern} we obtain that for every fixed point $P$,
$$
m_I(w_1(P),\ldots,w_n(P))x^{k}=Q_I\left(c_1^{S^1}(P),\ldots,c_n^{S^1}(P)\right),
$$
where $x$ is the variable in \eqref{restriction p}.
Then, by the localization formula \eqref{ABBV}, the coefficient of $G_{I,N}(\tau)$ can be obtained from the integral on $\M$ of the equivariant cohomology class
given by $Q_I(c_1^{S^1},\ldots,c_n^{S^1})\in H_{S^1}^{2k}(\M;\Z)$, namely
\begin{equation}\begin{split}\label{Q integral}
q_I(\M)&:= \int_\M Q_I\left(c_1^{S^1},\ldots,c_n^{S^1}\right)\\
& = \sum_{P \in \M^{S^{1}}}\frac{m_{I}\left(w_{1}(P),\dots,w_{n}(P)\right)}{w_{1}(P)\cdots w_{n}(P)}x^{k-n} \in H_{S^1}^{2\,(k-n)}(pt;\Z).
\end{split}\end{equation}

In the next subsection we give alternative formulas to compute explicitly these coefficients when $\M$ is a coadjoint orbit.
\subsection{Formulas to compute the coefficients \texorpdfstring{$q_I$}{Lg} for coadjoint orbits}\label{section coadjoint orbit}
In this subsection we give additional formulas to compute the coefficients $q_I(\M)$, if $\M$
is a coadjoint orbit. Before doing so we need to recall a few standard facts, whose proofs are omitted here (for more details and proofs we refer the reader to \cite{GHZ}, \cite[Section 4.2]{GSZ}
\cite[Section 6]{ST}, \cite[Section 5.3.1]{GHS} and the references therein).

Let $G$ be a compact simple Lie group with Lie algebra $\mathfrak{g}$. Choose a maximal torus $\T\subset G$, and denote by $\mathfrak{t}$ its Lie algebra.
We can embed $\mathfrak{t}^*$ in $\mathfrak{g}^*$ by means of a positive definite symmetric $G$-invariant linear form on $\mathfrak{g}$.
Let $P_0\in \mathfrak{t}^*$ and consider the orbit of $P_0$ in $\mathfrak{g}^*$ under the coadjoint action, $\mathcal{O}_{P_0}=G\cdot P_0$. Then it is well-known that
$\mathcal{O}_{P_0}$ can be endowed with the so-called Kostant-Kirillov symplectic form $\omega$. The natural action of the maximal torus $\T$ on $\mathcal{O}_{P_0}$
is indeed Hamiltonian with moment map given by the composition of the inclusion $\mathcal{O}_{P_0}\hookrightarrow \mathfrak{g}^*$ followed by the projection $\mathfrak{g}^* \to \mathfrak{t}^*$. Moreover the fixed point set is discrete. In order to describe the fixed point set data, which includes the weights, we need to introduce  more terminology.

Let $R\subset \mathfrak{t}^*$ be the set of roots, $R^+$ a choice of positive roots and $R_0\subset R^+$ the simple roots. The Weyl group $W$ of $G$ is generated by the reflections
through the hyperplanes orthogonal to the simple roots; we denote such reflections by $s_{\alpha}\colon \mathfrak{t}^* \to \mathfrak{t}^*$, with $\alpha\in R_0$, and the hyperplane orthogonal to $\alpha$ by $H_{\alpha}$. Moreover, if $P\in \mathfrak{t}^*$ and $w\in W$, we denote the point $P$ ``moved'' by $w$ with $w(P)$.
 Let $J$ be a (possibly empty) subset of simple roots; denote by $\langle J\rangle$ the set of positive roots that can be expressed as linear combinations of
roots in $J$ and by $W_J$ the subgroup of the Weyl group generated by reflections $s_{\alpha}$ with $\alpha \in J$.
Suppose that the chosen point $P_0\in \mathfrak{t}^*$ above lies in the (possibly empty) intersection of hyperplanes $\cap_{\alpha \in J}H_{\alpha}$ and is generic in this intersection, meaning that $s_{\alpha}(P_0)\neq P_0$ for all $\alpha\in R^+ \setminus \langle J \rangle$.
Denote by $W/W_J$ the set of right cosets, namely
$
W/W_J :=\{w \, W_J: w\in W\}
$, and by $[w]$ its elements. Then the following lemma holds (see in particular \cite[Section 5.3.1]{GHS})
\begin{lemma}\label{coadjoint fixed}
The map
$
W/W_J\to \mathfrak{t}^*
$
assigning $w(P_0)$ to $[w]$ is well-defined and, if restricted to the image, defines a bijection between $W/W_J$ and the fixed point set $\mathcal{O}_{P_0}^\T$.
Moreover the set of weights at $w(P_0)$ is given by $w(R^+\setminus \langle J \rangle)=\{w(\alpha): \alpha \in R^+ \setminus \langle J \rangle\}$.
\end{lemma}
Note that the above lemma also implies that the (real) dimension of $\mathcal{O}_{P_0}$ is $2 |R^+ \setminus \langle J \rangle|$.

With Lemma \ref{coadjoint fixed} it is in principle possible to compute the coefficients $q_I(\M)$, if $\M$ is a coadjoint orbit. However there is an alternative formula
that makes use of the divided difference operators, which is given in the proposition below. Before stating it we recall the following.
Denote by $\mathbb{S}(\mathfrak{t}^*)$ the symmetric algebra on $\mathfrak{t}^*$, which can be identified with the polynomials with complex coefficients
in the variables $x_1,\ldots,x_m$, where $\{x_1,\ldots,x_m\}$ is a $\Z$-basis of the dual of the integral lattice of $\mathfrak{t}^*$. We extend the $W$-action on $\mathfrak{t}^*$
to an action on the elements of  $\mathbb{S}(\mathfrak{t}^*)$ in the natural way. Let $\alpha_1,\ldots,\alpha_m \in R_0$ be the simple roots and for every
$j\in\{1,\ldots,m\}$ define $\partial_j \colon \mathbb{S}(\mathfrak{t}^*) \to  \mathbb{S}(\mathfrak{t}^*)$ to be
$$
\partial_j P := \frac{P-s_j(P)}{\alpha_j}\,,
$$
where $s_j$ is the simple reflection $s_{\alpha_j}$. This is called the \emph{divided difference operator} associated to $s_j:=s_{\alpha_j}$.
Given $w\in W$ and a reduced expression of $w$ in terms of simple roots, $w=s_{j_1}s_{j_2}\cdots s_{j_l}$, we define
$\partial_w := \partial_{j_1}\circ\partial_{j_2}\circ\cdots \circ\partial_{j_l}$. It is easy to see that such map does not depend on the reduced expression chosen for $w$, and is therefore called
the \emph{divided difference operator associated to }$w$.

We also recall that on the elements of the Weyl group there is a well-defined \emph{length function}, namely given a reduced expression of $w$ in terms of simple reflections
$s_{j_1}s_{j_2}\cdots s_{j_l}$, the number of simple reflections involved is independent on the reduced expression chosen for $w$, is called the {\it length} of $w$
and is denoted by $l(w)$. Moreover, given a right coset $[w]=\{ww': w'\in W_J\}$, we recall that there is a unique element in $[w]$ with minimal length --called \emph{minimal coset representative}-- and we can define
$l([w])$ as the length of that element. Finally, there is a unique equivalence class $[\overline{w}]$ in $W/W_J$ with maximal length, and this length coincides with $|R^+\setminus \langle J \rangle|$ (see \cite[Sections 1.6--1.10]{Hum}). Thus we have
$$
l([\overline{w}])=|R^+\setminus \langle J \rangle|=\frac{\dim\left(\mathcal{O}_{P_0}\right)}{2}=: n.
$$

\begin{prop}\label{formula divided}
Let $\mathcal{O}_{P_0}$ be a coadjoint orbit of a compact simple Lie group $G$, where $P_0$ is a generic point in the intersection of hyperplanes
$\cap_{\alpha\in J} H_\alpha$ for some $J\subset R_0$. Let $2n$ be the dimension of $\mathcal{O}_{P_0}$.
Denote by $[\overline{w}]$ the element with maximal length in $W/W_J$, where $\overline{w}$ is a minimal coset
representative. Then for every $I\in P_n(k)$
we have that
$$
q_I(\mathcal{O}_{P_0})= \partial_{\overline{w}} \,m_I\left(R^+ \setminus \langle J \rangle\right)\,,
$$
where $m_I$ is the monomial symmetric polynomial in the variables given by the roots in $R^+\setminus \langle J \rangle$.
\end{prop}

\begin{proof}
The proof of this fact makes use of the $\T$-equivariant cohomology ring structure of $\mathcal{O}_{P_0}$ and in particular of the so-called invariant classes and the equivariant Schubert classes (see \cite[Section 6]{ST} and \cite[Subsection 6.1]{GSZ}).

We begin with the case in which $J=\emptyset$. Then $P_0$ is a generic point in $\mathfrak{t}^*$ and the coadjoint orbit is also called generic.
Let us recall what invariant classes are. Given a polynomial $f\in \mathbb{S}(\mathfrak{t}^*)$ of degree $P$ we want to produce an element of
$H_\T^{2P}(\mathcal{O}_{P_0})$, called  invariant class (associated to $f$), in the following way. Let $c_f$ be the map from $W$ to $\mathbb{S}(\mathfrak{t}^*)$
defined by $w \mapsto w(f)$. Then it can be checked (see \cite{GSZ}) that such a map belongs to the image of the pull-back in equivariant cohomology
induced by the inclusion $(\mathcal{O}_{P_0})^\T \hookrightarrow \mathcal{O}_{p_0}$
$$
H_\T^*\left(\mathcal{O}_{P_0};\C\right) \to H_\T^*\left(\left(\mathcal{O}_{P_0}\right)^\T;\C\right)= \bigoplus_{P\in \left(\mathcal{O}_{P_0}\right)^\T} H^*_\T\left(pt;\C\right)= \bigoplus _{w\in W}\mathbb{S}\left(\mathfrak{t}^*\right).
$$
Abusing notation, we denote the corresponding element in $H_\T^*(\mathcal{O}_{P_0};\C)$ by $c_f$ too.
For instance, from \eqref{restriction chern} and Lemma \ref{coadjoint fixed} it is easy to see that the (restrictions to the fixed point set of) equivariant Chern classes are indeed invariant, and
hence all classes of the form $Q(c_1^\T,\ldots,c_n^\T)$, where $Q$ is a polynomial. In particular the class $Q_I(c_1^\T,\ldots,c_n^\T)$ is an invariant class, where $Q_I$ is the polynomial
associated to $m_I$ defined in \eqref{Q I}. By \eqref{Q integral} we need to find the integral of this class on $\mathcal{O}_{P_0}$.
For this purpose we use the (special) expression of an invariant class in terms of equivariant Schubert classes. We first recall what the latter are.
For every $w\in W$, the equivariant Schubert class at $w$, denoted by $\tau_w$, is characterized to be the unique equivariant cohomology
class in $H_\T^{2l(w)}(\mathcal{O}_{P_0};\C)$ satisfying\\
(i) $\tau_w(w)=\prod \{\alpha\in R^+: w^{-1}(\alpha)\in -R^+\}$;\\
(ii) $\tau_w(w')=0$ for all $w'\in W\setminus \{w\}$ such that $l(w')\leq l(w)$.
\\
\noindent
It is well-known that the set of these classes forms a basis of $H_\T^*(\mathcal{O}_{P_0};\C)$ as a module over $\mathbb{S}(\mathfrak{t}^*)$. Thus for every class
$c\in H_\T^*(\mathcal{O}_{P_0};\C)$ and every $w\in W$ there exists a polynomial $a_w\in \mathbb{S}(\mathfrak{t}^*)$ such that $c=\sum_{w\in W} a_w \tau_w$.
The coefficients $a_w$ have a special expression in the case in which $c$ is an invariant class. Indeed, given $f\in \mathbb{S}(\mathfrak{t}^*)$ and the corresponding invariant class
$c_f$, Proposition 6.1 in \cite{GSZ} asserts that $a_w=(-1)^{l(w)}\partial_w f$.

To conclude the proof of the proposition we note that $\int_{\mathcal{O}_{P_0}}\tau_w=0$ unless $w$ is the longest element $\overline{w}$ in $W$, in which case the localization formula in equivariant cohomology gives that $\int_{\mathcal{O}_{P_0}}\tau_{\overline{w}}=(-1)^{l(\overline{w})}$. (This is because the set of weights $\{\alpha\in R^+: \overline{w}^{-1}(\alpha)\in -R^+\}$ coincides with minus
the isotropy weights at $\overline{w}$, and the restriction of the class $\tau_{\overline{w}}$ at the elements of $W\setminus\{\overline{w}\}$ is zero.)
Hence, since the invariant class $Q_I(c_1^\T,\ldots,c_n^\T)$ at $P_0$ is given by $m_I(R^+)$,
equation \eqref{Q integral} gives
$$
q_I\left(\mathcal{O}_{P_0}\right)=\int_{\mathcal{O}_{P_0}} Q_I\left(c_1^\T,\ldots,c_n^\T\right)=\sum_{w\in W} (-1)^{l(w)}\partial_w m_I(R^+)\int_{\mathcal{O}_{P_0}}\tau_w=\partial_{\overline{w}} \,m_I\left(R^+\right).
$$
We only hint at the proof of this proposition in the case in which $P_0$ is not generic: for partial coadjoint orbits, equivariant Schubert classes are also defined and mutatis mutandis satisfy properties (i) and (ii) above. Moreover
\cite[Proposition 6.1]{GSZ} admits a generalization to partial coadjoint orbits, and the claim follows similarly.
\end{proof}

\begin{exm}\label{example coadjoint} \ \\
\noindent
(1) For $\mathrm{SU}(n+1)$ the root system is of type $A_n$ and a choice of simple roots is given by $R_0=\{x_1-x_2,x_2-x_3,\ldots,x_n-x_{n+1}\}$.
If $J=R_0\setminus\{x_1-x_2\}$ it is well-known that the corresponding coadjoint orbit is $\C P^n$. In this case the longest equivalence class in $W/W_J$ has minimal
coset representative given by $s_ns_{n-1}\cdots s_1$, where $s_j:=s_{x_j-x_{j+1}}$ for all $j\in\{1,\ldots,n\}$; moreover $R^+\setminus \langle J \rangle=\{x_1-x_2,x_1-x_3,\ldots,x_1-x_n\}$. Hence for every partition $I\in P_n(k)$ Proposition \ref{formula divided} gives
$$
q_I(\C P^n)= \partial_n \partial_{n-1} \cdots  \partial_1 m_I(x_1-x_2,x_1-x_3,\ldots,x_1-x_n)\,.
$$
\\
\noindent
(2) For $\mathrm{SO}(2n+1)$ the roots system is of type $B_n$ and a choice of simple roots is given by $R_0=\{x_1-x_2,\ldots,x_{n-1}-x_n\}\cup \{x_n\}$.
If $J=R_0\setminus\{x_1-x_2\}$ it is well-known that the corresponding coadjoing orbit is $\mathrm{Gr}_2^+(\R^{2n+1})$, the Grassmannian of oriented two planes in $\R^{2n+1}$.
This is a symplectic manifold of dimension $2(2n-1)$.
Let $s_j$ be $s_{x_j-x_{j+1}}$ for all $j\in\{1,\ldots,n-1\}$ and $s_n:=s_{x_n}$. Then the longest equivalence class in $W/W_J$ has minimal coset representative given by
$s_1\cdots s_{n-1}s_ns_{n-1}\cdots s_1$ and $R^+\setminus \langle J \rangle = \{x_1-x_2,\ldots,x_1-x_n,x_1+x_2,\ldots,x_1+x_n,x_1\}$. Then
Proposition \ref{formula divided} gives in this case that for every partition $I\in P_{2n-1}(k)$
\begin{align*}
&q_I\left(\mathrm{Gr}_2^+\left(\R^{2n+1}\right)\right)\\
&=\partial_1\cdots \partial_{n-1}\partial_n\partial_{n-1}\cdots \partial_1 m_I(x_1-x_2,\ldots,x_1-x_n,x_1+x_2,\ldots,x_1+x_n,x_1)\,.
\end{align*}
\end{exm}

\subsection{Explicit computation of relations among Eisenstein series}\label{explicit}
In this subsection we show that the manifold $\C P^n$ and the
vanishing of its elliptic genus $\varphi_{N}(\C P^{n},t)$, for every
$N \mid (n+1)$, can be used to derive explicit relations among
Eisenstein series for every $n\geq 1$.

    \begin{prop}\label{proposition general relation CPn}
            For $N\geq 2$ dividing  $(n+1)$ and $k \geq n$ we have
            \begin{align*}
            &(-1)^{n+k+1}\sum_{\substack{0 \leq j_{1},\dots,j_{n} \leq k \\ j_{1} + \ldots + j_{n} =k}}G_{j_{1},N}\cdots G_{j_{n},N} =
                \sum_{\ell = 0}^{n-1}\binom{k-\ell-1}{n-\ell-1}G_{k-\ell,N}\sum_{\substack{0 \leq j_{1},\dots,j_{n}\leq \ell \\ j_{1}+\ldots +j_{n} = \ell}}G_{j_{1},N}\cdots G_{j_{n},N}.
            \end{align*}
    \end{prop}
\begin{proof}
As mentioned above, we want to use the vanishing of the elliptic
genus $\varphi_{N}(\C P^{n},t)$, for every $N\geq 2$ dividing
$(n+1)$ (for the vanishing statement see \cite{HBJ} and Proposition
\ref{toric vanishing}, since $\C P^n$ can be endowed with a toric
action), to derive the above relations. In order to compute the
coefficients $q_I(\C P^n)$ we have two possibilities: one is to use
directly the localization formula \eqref{Q integral}, and another
one is to use Proposition \ref{formula divided} (see Example
\ref{example coadjoint} (1)). In this proposition, we first used the localization
formula and a trick involving a residue calculation to guess what the coefficients are. After guessing them, we found a number theoretical proof of the above identity which is given below.
 Note that, however, the use of the rigidity theorem, and in particular of Theorem \ref{main number theory}, plays the essential role of finding what the right coefficients in the identity are.

 By definition, the power series for the elliptic genus of level $N$ is given by
            \[
            \mathcal Q_{N}(x) = x\eta^{3}\frac{\vartheta\left(\frac{x}{2\pi i}-\frac{1}{N} \right)}{\vartheta\left(\frac{x}{2\pi i} \right)\vartheta\left(-\frac{1}{N}\right)} = \sum_{j=0}^{\infty}a_{j}x^{j},
            \]
            where we omit the $\tau$-variable for brevity. Recall that $a_{j} = G_{j,N}$ is an Eisenstein series by Lemma~\ref{lemma eisenstein}. For $k\geq n$ we define the power series
            \[
            P_{n,k}(x) = \frac{1}{(k-n)!}x^{k-n+1}\frac{d^{k-n}}{dx^{k-n}}\frac{\mathcal Q_{N}(x)}{x} = (-1)^{k+n}+\sum_{j=k-n+1}^{\infty}\binom{j-1}{k-n}a_{j}x^{j}.
            \]
            Then the formula in the proposition is equivalent to showing that the coefficient at $x^{k}$ of the power series $P_{n,k}(x)\mathcal Q_{N}(x)^{n}$ vanishes. We have
            \[
            (k-n)!P_{n,k}(x)\mathcal Q_{N}(x)^{n} = x^{k+1} \left(\frac{d^{k-n}}{dx^{k-n}}\frac{\mathcal Q_{N}(x)}{x}\right)  \left(\eta^{3}\frac{\vartheta\left(\frac{x}{2\pi i}-\frac{1}{N} \right)}{\vartheta\left(\frac{x}{2\pi i} \right)\vartheta\left(-\frac{1}{N}\right)} \right)^{n}
            \]
            and thus it suffices to show that the residue at $x = 0$ of the function
            \[
            \left(\frac{d^{k-n}}{dx^{k-n}}\frac{\vartheta\left(x-\frac{1}{N} \right)}{\vartheta\left(x \right)}\right)  \left(\frac{\vartheta\left(x-\frac{1}{N} \right)}{\vartheta\left(x \right)} \right)^{n}
            \]
            vanishes. Using the elliptic transformation behavior \eqref{eq theta elliptic transformations} it is easy to check that the above function is invariant under the lattice $\Z \tau + \Z$ if $N \mid (n+1)$. Hence the sum of its residues in a fundamental domain for $\C/\left(\Z \tau +\Z\right)$ vanishes. But the function also has a unique pole at $x = 0$, hence its residue there has to vanish. This completes the proof.
\end{proof}

        The fact that the above proposition has a simple proof seems to correspond to the fact that $\C P^{n}$ is very symmetric. However, the rigidity of the elliptic genus yields many more relations of Eisenstein series which do not seem to have such simple proofs.

\end{document}